\documentclass[reqno,11pt]{amsart}

\usepackage{amsmath,amsfonts,amssymb,graphicx,amsthm,enumerate,url,tikz,mathrsfs}
\usepackage[noadjust]{cite}
\usepackage{stmaryrd}

\usepackage{caption}
\usepackage[colorlinks=true]{hyperref}

\numberwithin{equation}{section}

\renewcommand{\epsilon}{\varepsilon}
\newcommand{\given}{\;\big|\;}
\newcommand{\one}{\mathbf{1}}

\usepackage{color}
 \definecolor{refkey}{gray}{.5}
 \definecolor{labelkey}{gray}{.5}
\definecolor{light}{gray}{.9}

\newtheorem{theorem}{Theorem}[section]
\newtheorem{lemma}[theorem]{Lemma}

\newtheorem{corollary}[theorem]{Corollary}
\newtheorem{remark}[theorem]{Remark}

\newtheorem{claim}[theorem]{Claim}

\newtheorem*{definition*}{Definition}
\newtheorem{maintheorem}{Theorem}

\newtheorem{mainprop}[maintheorem]{Proposition}
\newtheorem*{question*}{Question}

\newtheorem*{remark*}{Remark}

\newcommand{\E}{\mathbb E}

\renewcommand{\P}{\mathbb P}

\newcommand{\cA}{\ensuremath{\mathcal A}}
\newcommand{\cB}{\ensuremath{\mathcal B}}
\newcommand{\cC}{\ensuremath{\mathcal C}}

\newcommand{\cE}{\ensuremath{\mathcal E}}
\newcommand{\cF}{\ensuremath{\mathcal F}}
\newcommand{\cG}{\ensuremath{\mathcal G}}

\newcommand{\cK}{\ensuremath{\mathcal K}}

\newcommand{\cN}{\ensuremath{\mathcal N}}

\newcommand{\sX}{{\ensuremath{\mathscr X}}}

\newcommand{\couple}{\Pi}

\newcommand{\dimH}{{\mathbf{d}}}
\newcommand{\spe}{\nu}
\newcommand{\tv}{\textsc{tv}}
\newcommand{\smP}{\textsc{p}}
\newcommand{\btr}{\cB_{\operatorname{tr}}}
\newcommand{\bcyc}{\cB_{\operatorname{cyc}}}
\newcommand{\bdeg}{\cB_{\operatorname{deg}}}
\def\eps{\varepsilon}

\DeclareMathOperator{\var}{Var}
\DeclareMathOperator{\dist}{dist}

\DeclareMathOperator{\cov}{Cov}

\DeclareMathOperator{\Poi}{Poisson}
\DeclareMathOperator{\Geom}{Geom}
\DeclareMathOperator{\Bin}{Bin}

\newcommand{\tmix}{t_{\textsc{mix}}}

\newcommand{\gwsurv}{\ensuremath{\mathrm{GW}^*}}
\newcommand{\agwsurv}{\ensuremath{\mathrm{AGW}^*}}
\newcommand{\gwharm}{\ensuremath{\mathrm{GW}_{\textsc{harm}}}}

\begin{document}

\title{Random walks on the random graph}

\author[N. Berestycki]{Nathana\"el Berestycki}
 \address{N. Berestycki\hfill\break
University of Cambridge, DPMMS, Wilberforce Rd. Cambridge CB3 0WB, UK.}
\email{n.berestycki@statslab.cam.ac.uk}

\author[E. Lubetzky]{Eyal Lubetzky}
\address{E.\ Lubetzky\hfill\break
Courant Institute\\ New York University\\
251 Mercer Street\\ New York, NY 10012, USA.}
\email{eyal@courant.nyu.edu}

\author[Y. Peres]{Yuval Peres}
\address{Y.\ Peres\hfill\break
Microsoft Research\\ One Microsoft Way\\ Redmond, WA 98052, USA.}
\email{peres@microsoft.com}

\author[A. Sly]{Allan Sly}
\address{A. Sly\hfill\break
Department of Statistics\\
UC Berkeley\\
Berkeley, CA 94720, USA.}
\email{sly@stat.berkeley.edu}

\begin{abstract}
We study random walks on the giant component of the Erd\H{o}s-R\'enyi random graph $\cG(n,p)$ where $p=\lambda/n$ for $\lambda>1$ fixed.
The mixing time from a worst starting point was shown by Fountoulakis and Reed, and independently by Benjamini,
Kozma and Wormald, to have order $\log^2 n$. We prove that starting from a uniform vertex (equivalently, from a fixed vertex conditioned to belong to the giant) both accelerates mixing to $O(\log n)$ and concentrates it (the cutoff phenomenon occurs): the typical mixing is at $(\spe \dimH)^{-1}\log n \pm (\log n)^{1/2+o(1)}$, where $\spe$ and $\dimH$ are the speed of random walk and dimension of harmonic measure on a $\Poi(\lambda)$-Galton-Watson tree. Analogous results are given for graphs with prescribed degree sequences, where cutoff is shown both for the simple and for the non-backtracking random walk.
\end{abstract}
{\mbox{}
\maketitle
}
\vspace{-0.55cm}
\section{Introduction}\label{sec:intro}

The time it takes random walk to approach its stationary distribution on a graph is a gauge for an array of properties of the underlying geometry: it reflects the histogram of distances between vertices, both typical and extremal (radius and diameter); it is affected by local traps (e.g., escaping from a bad starting position) as well as by global bottlenecks (sparse cuts between large sets); and it is closely related to the Cheeger constant and expansion of the graph. In this work we study random walk on the giant component $\cC_1$ of the classical Erd\H{o}s-R\'enyi random graph $\cG(n,p)$, and build on recent advances in our understanding of its geometry on one hand, and random walks on trees and on random regular graphs on the other, to provide sharp results on mixing on $\cC_1$ and on the related model of a random graph with a prescribed degree sequences.

The Erd\H{o}s-R\'enyi random graph $\cG(n,p)$ is the graph on $n$ vertices where each of the $\binom{n}2$ possible edges appears independently with probability $p$. In their celebrated papers from the 1960's, Erd\H{o}s and R\'enyi discovered the so-called ``double jump'' in the size of $\cC_1$, the largest component in this graph: taking $p=\lambda/n$ with $\lambda$ fixed, at $\lambda<1$ it is with high probability (w.h.p.) logarithmic in $n$; at $\lambda=1$ it is of order $n^{2/3}$ in expectation; and at $\lambda>1$ it is w.h.p.\ linear (a ``giant component''). Of these facts, the critical behavior was fully established only much later by Bollob\'as~\cite{Bollobas84} and {\L}uczak~\cite{Luczak90}, and extends to the \emph{critical window} $p=(1\pm\epsilon)/n$ for $\epsilon = O(n^{-1/3})$ as discovered in~\cite{Bollobas84}.

An important notion for the rate of convergence of a Markov chain to stationarity is its (worst-case) total-variation mixing time: for a transition kernel $P$ on a state-space $\Omega$ with a stationary distribution $\pi$, recall $\|\mu-\nu\|_{\tv}=\sup_{A\subset\Omega} [\mu(A)-\nu(A)]$ and write

\[ \tmix(\epsilon)=\min\{ t : d_{\tv}(t)<\epsilon\} \quad\mbox{ where }\quad d_{\tv}(t)=\max_{x}\|P^t(x,\cdot)-\pi\|_\tv\,;\]
let $\tmix^{(x)}(\epsilon)$ and $d_{\tv}^{(x)}(t)$ be the analogs from a fixed (rather than worst-case) initial state $x$.
The effect of the threshold parameter $0<\epsilon<1$ is addressed by the {\em cutoff phenomenon}, a sharp transition from $d_{\tv}(t)\approx 1$ to $d_{\tv}(t)\approx0$, whereby $\tmix(\epsilon)=(1+o(1))\tmix(\epsilon')$ for any fixed $0<\epsilon,\epsilon'<1$
(making $\tmix(\epsilon)$ asymptotically independent of this $\epsilon$).

In recent years, the understanding of the geometry of $\cC_1$ and techniques for Markov chain analysis became sufficiently developed to give the typical order of the mixing time of random walk, which transitions from $n$ in the critical window (\cite{NP08}) to $\log^2 n$ at $p=\frac{\lambda}n$ on for $\lambda>1$ fixed (\cite{BKW14,FR08}) through the interpolating order $\epsilon^{-3}\log^2(\epsilon^3 n)$ when  $p=\frac{1+\epsilon}n$ for $n^{-1/3}\ll\epsilon\ll 1$ (\cite{DLP12}).
Of these facts, the lower bound of $\log^2 n$ on the mixing time in the supercritical regime ($\lambda>1$) is easy to see, as $\cC_1$ w.h.p.\ contains a path of $c\log n$ degree-2 vertices for a suitable fixed $c>0$ (escaping this path when started at its midpoint would require $(\frac{c}2 \log n)^2$ steps in expectation).
The two works that independently obtained a matching $\log^2 n$ upper bound used very different approaches: Fountoulakis and Reed~\cite{FR08} relied on a powerful estimate from~\cite{FR07} on mixing in terms of the conductance profile (following Lov\'asz-Kannan~\cite{LK99}), while
Benjamini, Kozma and Wormald~\cite{BKW14} used a novel decomposition theorem of $\cC_1$ as a ``decorated expander.''

As the lower bound was based on having the initial vertex $v_1$ be part of a rare bottleneck (a path of length $c\log n$), one may ask what $\tmix^{(v_1)}$ is for a typical $v_1$. Fountoulakis and Reed~\cite{FR08} conjectured that if $v_1$ is not part of a bottleneck, the mixing time is accelerated to $O(\log n)$.  This is indeed the case for \emph{almost every} initial vertex $v_1$, and moreover $\tmix^{(v_1)}(\epsilon)$ then concentrates on $c_0\log n$ for $c_0$ independent of $\epsilon$ (cutoff occurs):

\begin{figure}[t]
\vspace{-0.5cm}
\centering
\raisebox{-0.25cm}{
\includegraphics[width=.55\textwidth]{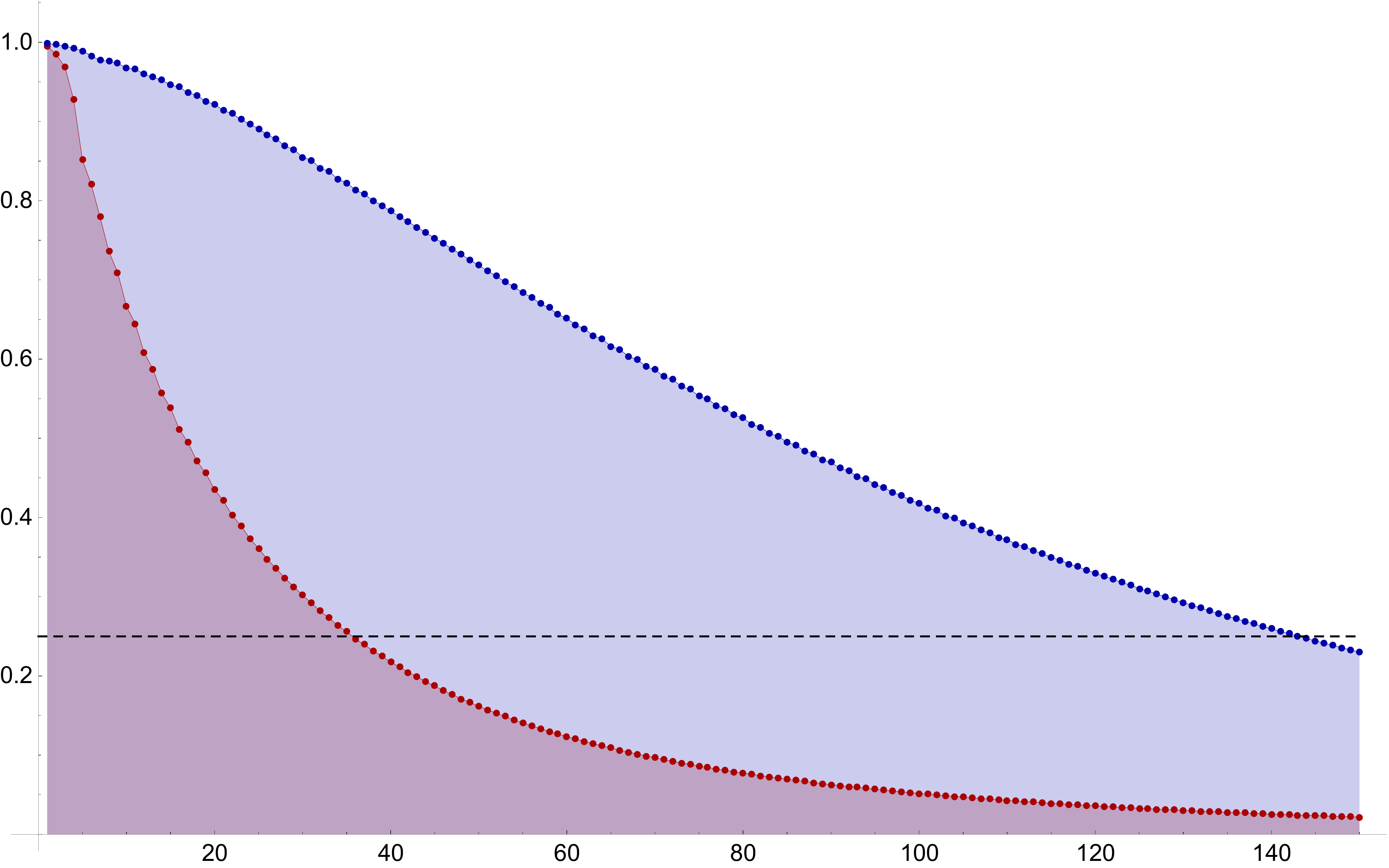}
}\hspace{-4.5cm}
\includegraphics[width=.7\textwidth]{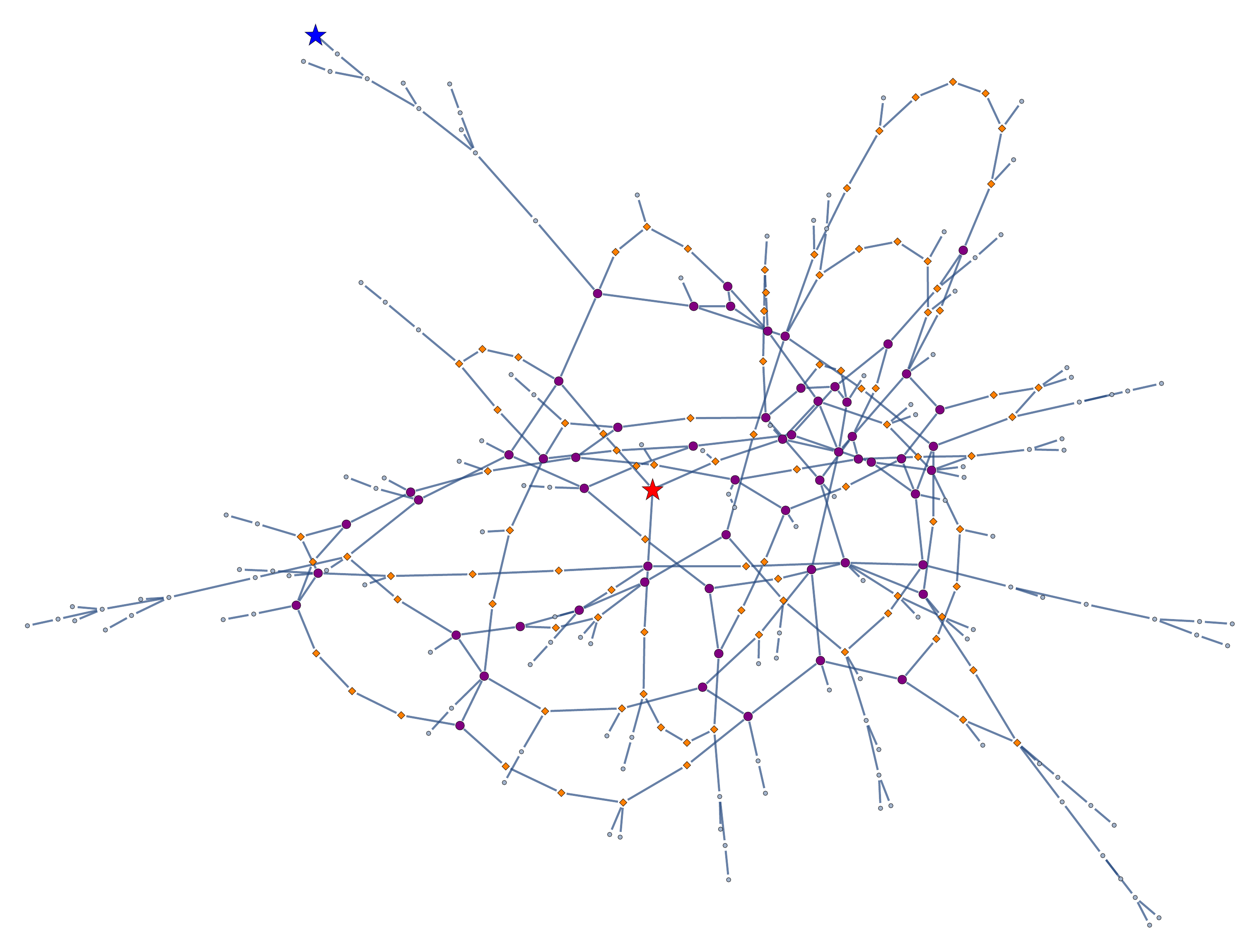}
\vspace{-0.3cm}
\caption{Total-variation distance to stationarity from typical/worst initial vertices (marked red/blue) on the giant component of $\cG(n,p=\frac2n)$.}
\label{fig:c1-tmix}
\vspace{-0.35cm}
\end{figure}

\begin{maintheorem}\label{mainthm-giant}
Let $\cC_1$ denote the giant component of the random graph $\cG(n,p=\lambda/n)$ for $\lambda>1$ fixed, and let $\spe $ and $\dimH$ denote the speed of random walk and the dimension of harmonic measure on a $\Poi(\lambda)$-Galton-Watson tree, resp.
For any $0<\epsilon<1$ fixed, w.h.p.\ the random walk from a uniformly chosen vertex $v_1\in\cC_1$ satisfies \begin{equation}\label{eq-tmix-giant}
\left| \tmix^{(v_1)}(\epsilon)- (\spe \dimH)^{-1} \log n\right| \leq \log^{1/2+o(1)}n\,.
\end{equation}
In particular, w.h.p.\ the random walk from $v_1$ has cutoff with a $\log^{1/2+o(1)}n$ window.
\end{maintheorem}

Before we examine the roles of $\spe$ and $\dimH$ in this result, it is helpful to place it in the context of the structure theorem for $\cC_1$, recently given in~\cite{DLP12} (see Theorem~\ref{thm-struct} below): a contiguous model for $\cC_1$ is given by (i) choosing a \emph{kernel} uniformly over graphs on degrees i.i.d.\ Poisson truncated to be at least 3; (ii) subdividing every edges via i.i.d.\ geometric variables; and (iii) hanging i.i.d.\ Poisson Galton-Watson (GW) trees on every vertex.
Observe that Steps (ii) and (iii) introduce i.i.d.\ delays with an exponential tail for the random walk; thus, starting the walk from a uniform vertex (rather than on a long path or a tall tree) would essentially eliminate all but the typical $O(1)$-delays, and it should rapidly mix on the kernel (w.h.p.\ an expander) in time $O(\log n)$; see Fig.~\ref{fig:c1-tmix}.

\begin{figure}[t]
\begin{center}
  \begin{tikzpicture}[font=\tiny]
    \newcommand{\hsep}{7cm}
    \node (plot1) at (0,0) {
      \includegraphics[width=.45\textwidth]{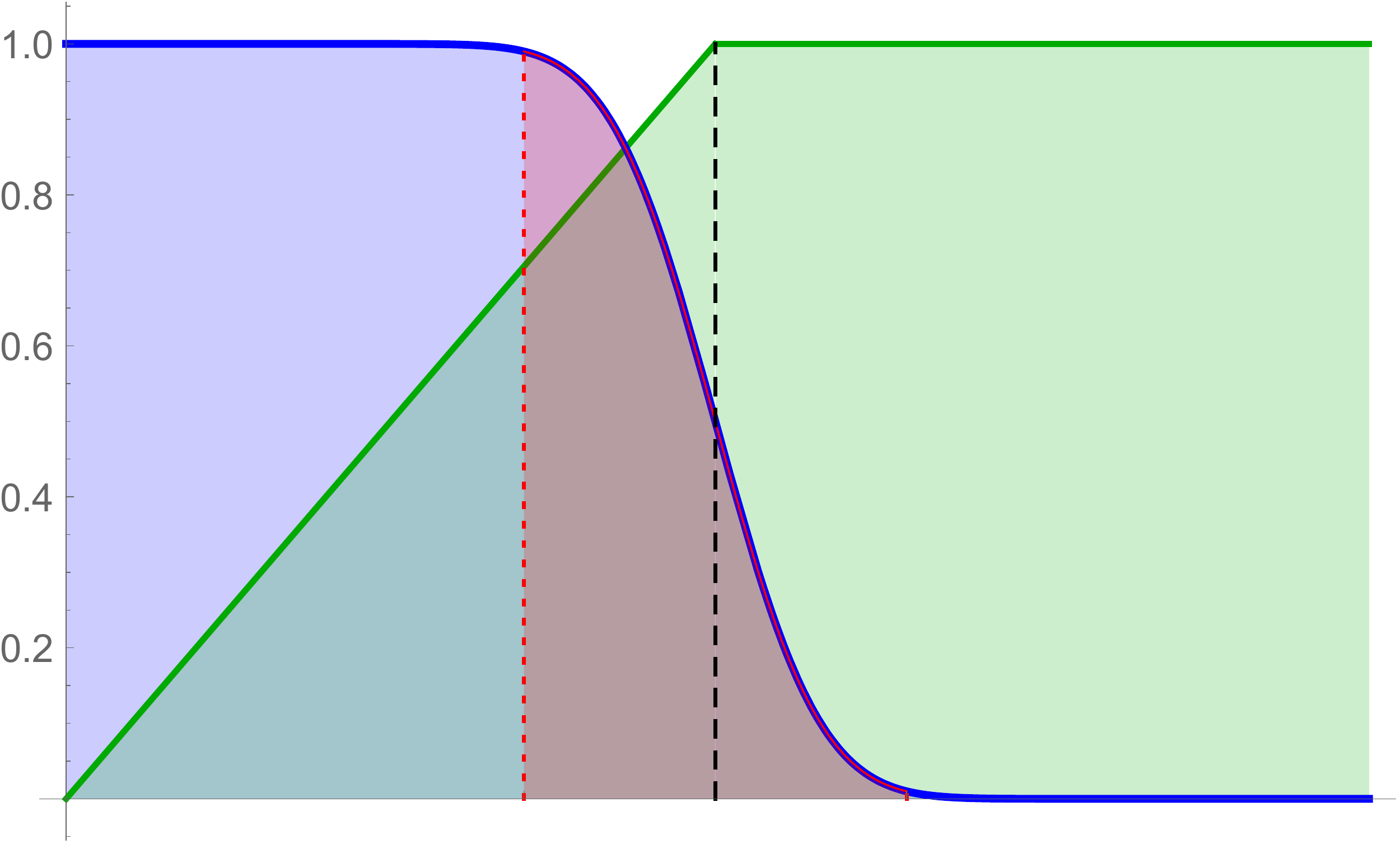}};
    \node (plot2) at (\hsep,0) {
      \includegraphics[width=.45\textwidth]{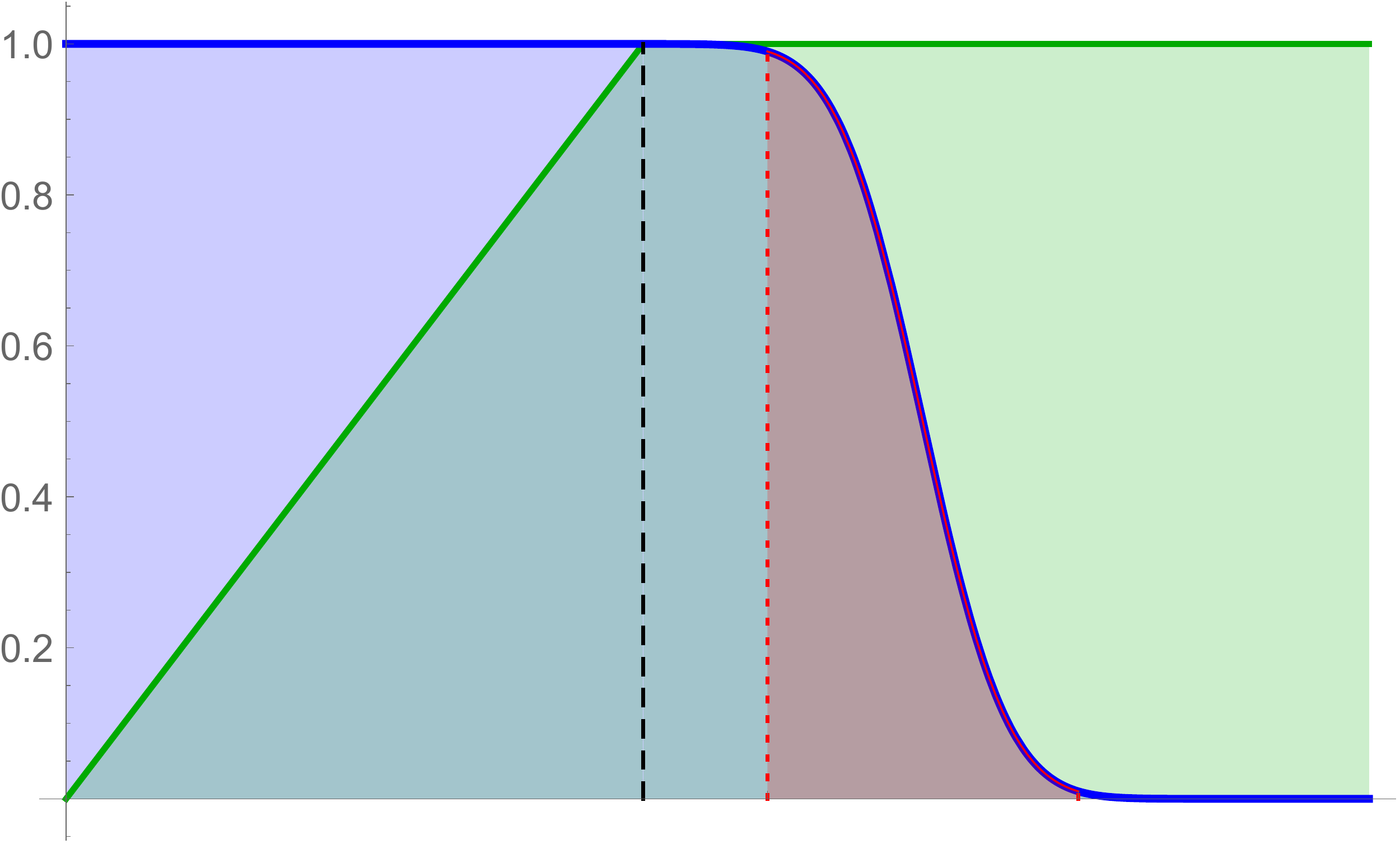}};
    \begin{scope}[shift={(plot1.south west)}]
    \node at (2.4,0.2) {$\tmix(1-\epsilon)$};
    \node at (3.6,0.2) {$\tfrac1\nu\bar{D} $};
    \node at (5,0.2) {$\tmix(\epsilon)$};
    \end{scope}
    \begin{scope}[shift={(plot2.south west)}]
    \node at (3,0.2) {$\tfrac1\nu\bar{D} $};
    \node at (4.2,0.2) {$\tmix(1-\epsilon)$};
    \node at (5.5,0.2) {$\tmix(\epsilon)$};
    \end{scope}
  \end{tikzpicture}
\end{center}
\vspace{-0.25cm}
\caption{Mixing in total-variation vs.\ distance from the origin (rescaled to $[0,1]$). On  left: random walk on a random regular graph mixes once its distance from the origin reaches the typical distance $\bar{D}$. On right: the analog on a random non-regular graph, where mixing is further delayed due to the dimension drop of harmonic measure.}
\label{fig:dist-mix}
\vspace{-0.4cm}
\end{figure}
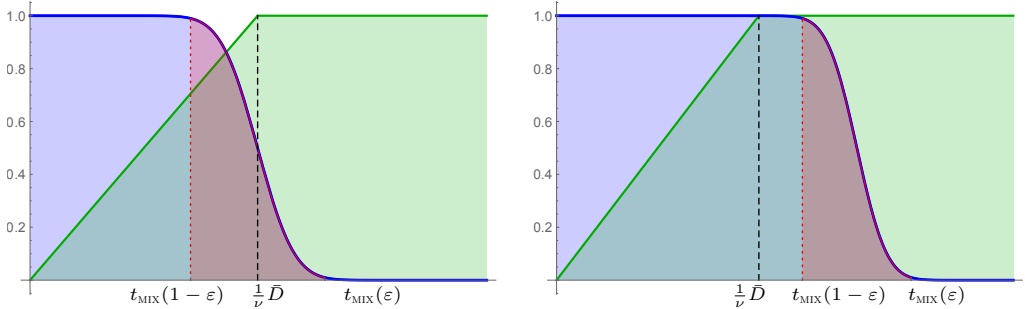

It is well-known (see~\cite{Bollobas-RG,JLR-RG,Durrett-RGD,RW10}) that $\bar{D}$, the average distance between two vertices in $\cC_1$, is $\log_\lambda n+O_{\smP}(1)$, analogous to the fact that $\bar{D}=\log_{d-1} n + O_{\smP}(1)$ in $\cG(n,d)$, the uniform $d$-regular graph on $n$ vertices\footnote{Here $O_{
\smP}(1)$ denotes a random variable that is bounded in probability.} (both are locally-tree-like: $\cG(n,d)$ resembles a $d$-regular tree while $\cG(n,p)$ resembles a $\Poi(\lambda)$-GW tree). It is then natural to expect that $\tmix^{(v_1)}$ coincides with the time it takes the walk to reach this typical distance from its origin $v_1$, which would be $\nu^{-1}\bar{D}$ for a random walk on a $\Poi(\lambda)$-GW tree

 Supporting evidence for this on the random 3-regular graph $\cG(n,3)$ was given in~\cite{BD08}, where it was shown that the distance of the walk from $v_1$
 after $t=c\log n$ steps is w.h.p.\ $(1+o(1))(\nu t \wedge \bar{D})$, with $\nu=\frac13$ being the speed of random walk on a binary tree.
  Durrett~\cite[\S6]{Durrett-RGD}
  conjectured that reaching the correct distance from $v_1$ indicates mixing, namely that $\tmix(\frac14)\sim 2 \nu^{-1}\bar{D}$ for the lazy (hence the extra factor of 2) random walk on $\cG(n,3)$. This was confirmed in~\cite{LS10}, and indeed on $\cG(n,d)$ the simple random walk has $\tmix(\epsilon)=\frac{d}{d-2}\log_{d-1} n+O(\sqrt{\log n})$, {\it i.e.}, there is cutoff at $\nu^{-1}\bar{D}$ with an $O(\sqrt{\log n})$ window (in particular, random walk has cutoff on almost every $d$-regular; prior to the work~\cite{LS10} cutoff was confirmed almost exclusively on graphs with unbounded degree).

However, Theorem~\ref{mainthm-giant} shows that $\tmix\sim (\nu\dimH)^{-1} \log n$ vs.\ the $\nu^{-1}\log_\lambda n$ steps needed for the distance from $v_1$ to reach its typical value (and stabilize there; see~Corollary~\ref{cor-typical-distance}).
As it turns out, the ``dimension drop'' of harmonic measure
(whereby $\dimH<\log\E Z$ unless the offspring distribution $Z$ is a constant), discovered in~\cite{LPP95}, plays a crucial role here, and stands behind this slowdown factor of $(\dimH/\log \lambda)^{-1}>1$. Indeed, while
generation $k$ of the GW-tree has size about $(\E Z)^k$, random walk at distance $k$ from the root concentrates on an exponentially small subset of size about $\exp(\dimH k)$
 (see Fig.~\ref{fig:dist-mix} and~\ref{fig:dim}).
Hence, $(\nu\dimH)^{-1}\log n$ is certainly a lower bound on $\tmix$ (the factor $\spe^{-1}$ translates time to the distance $k$ from the root), and Theorem~\ref{mainthm-giant} shows this bound is tight on $\cG(n,p)$.

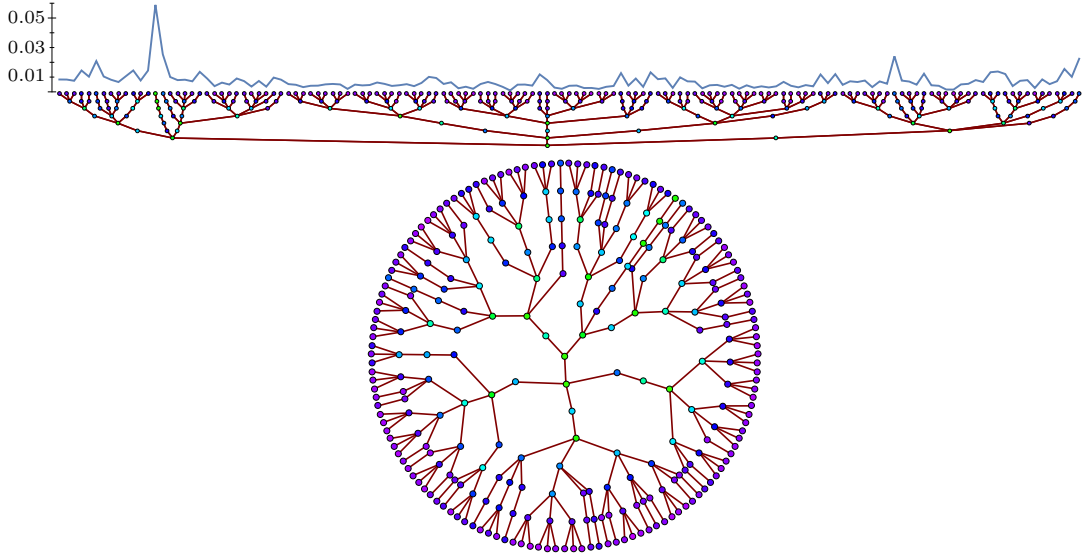
\begin{figure}[t]
\vspace{-0.4cm}
\begin{center}
  \begin{tikzpicture}[font=\tiny]
    \node (plot1) at (0,0) {
      \includegraphics[width=.45\textwidth]{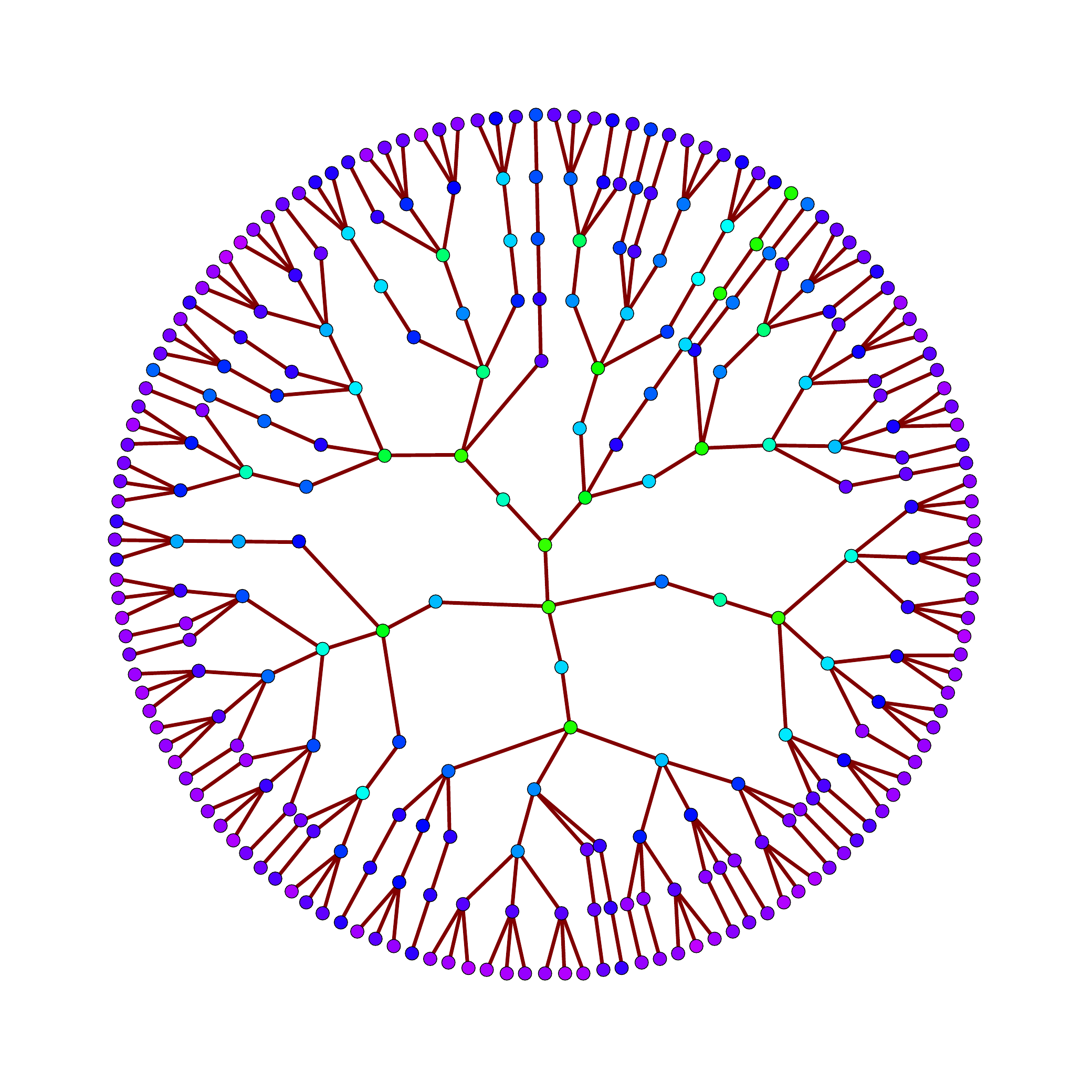}};
    \node (plot2) at (0,3.75) {
      \includegraphics[width=.95\textwidth]{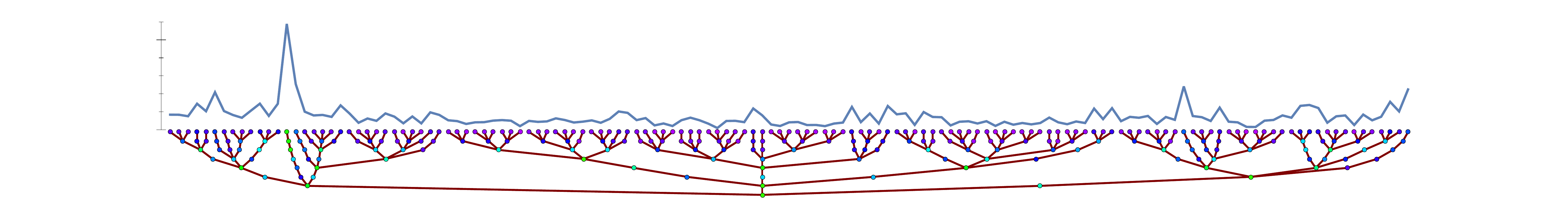}};
   \begin{scope}[shift={(plot2.north west)}]
   \node at (0.25,-0.35) [anchor=east] {0.05};
   \node at (0.25,-0.75) [anchor=east] {0.03};
   \node at (0.25,-1.12) [anchor=east] {0.01};
    \end{scope}
  \end{tikzpicture}
\end{center}
\vspace{-1.1cm}
\caption{Harmonic measure on the first 7 generations of a GW-tree with offspring distribution $\P(Z=1)=\P(Z=3)=\frac12$.
}
\label{fig:dim}
\vspace{-0.3cm}
\end{figure}

The same phenomenon occurs more generally in a random walk on a random graph with a degree distribution $(p_k)_{k=1}^\infty$, generated by first sampling the degree of each vertex $v$ via an i.i.d.\ random variable $D_v$ with $\P(D_v=k)=p_k$ conditioned on $\sum_v D_v$ being even,  then choosing the graph uniformly over all graphs with these prescribed degrees.

\begin{maintheorem}\label{mainthm:deg-seq}
Let $G$ be a random graph with degree distribution $(p_k)_{k=1}^\infty$, such that for some fixed $\delta>0$, the random variable $Z$ given by $\P(Z=k-1) \propto  k p_k $ satisfies
\begin{equation}\label{eq-Z-max-degree-hypo}
 \P(Z > \Delta_n) = o(1/n) \quad\mbox{ for }\quad \Delta_n := \exp\big[(\log n)^{1/2-\delta}\big]\,,
 \end{equation}
 and let
$ t_\star = (\spe \dimH)^{-1} \log n $,
where $\spe$ and $\dimH$ are the speed of random walk and dimension of harmonic measure on a Galton-Watson tree with offspring distribution $Z$.
\begin{enumerate}[(i)]

 \item \label{it-deg-seq-2}
Set $w_n = \sqrt{\log n}(\log\log n)^3$. If  $p_2 < 1-\delta$ and $ 1+\delta < \E Z < K$ for an absolute constant $K$,
then w.h.p.\ on the event that $v_1$ is in the largest component of $G$,
\[
 d^{(v_1)}_{\tv}\big(t_\star-w_n\big)> 1-\epsilon\,,\qquad \mbox{whereas}\qquad d^{(v_1)}_{\tv}\big(t_\star+ w_n\big)<\epsilon\,.\]
\item \label{it-deg-seq-1}
Set $w_n = \sqrt{\log n}$. If $Z \geq 2$ and $\E Z < K$ for some absolute constants $K$, then
for any $\epsilon>0$ there exists some $\gamma>0$ such that, with probability at least $1-\epsilon-o(1)$,
\[
 d^{(v_1)}_{\tv}\big(t_\star-\gamma w_n\big)> 1-\epsilon\,,\qquad \mbox{whereas}\qquad d^{(v_1)}_{\tv}\big(t_\star+\gamma w_n\big)<\epsilon\,.\]
\end{enumerate}
\end{maintheorem}

Condition~\eqref{eq-Z-max-degree-hypo} is weaker than requiring $Z$ to have finite exponential moments.

The intuition behind these results is better seen for the \emph{non-backtracking} (as opposed to simple) random walk (NBRW), which, upon arriving to a vertex $v$ from some other vertex $u$, moves to a uniformly chosen neighbor $w\neq u$ (formally, this is a Markov chain whose state-space is the set of directed edges in the graph).
This walk has speed $\nu=1$ on a GW-tree (as it never backtracks towards the root), and on $\cG(n,3)$ it was shown in~\cite{LS10} to satisfy $|\tmix(\epsilon) - \log_{2} n| < C(\epsilon)$ for any fixed $0<\epsilon<1$ w.h.p.---indeed,   cutoff (with an $O(1)$-window) occurs once the distance from the origin reaches the average graph distance. If we instead take a random graph on $2n/3$ vertices of degree 2 and $n/3$ vertices of degree 4, this corresponds to a GW-tree with an offspring distribution $\P(Z=1)=\P(Z=3)=\frac12$; since its $k$-th generation grows as $2^k$, the distance between two typical vertices is again asymptotically $\log_{2} n$. However, the probability that the NBRW follows a given path $v_1,v_2,\ldots,v_k$ is $\prod 1/Z_i$ (with $Z_i$ denoting the number of children of $v_i$); setting $\dimH = \E \log Z = \log \sqrt{3} < \log 2$, observe that $\sum_{i<k} \log Z_i$ concentrates around $k \dimH$ by CLT, and hence $k \sim \dimH^{-1}\log n$ is a lower bound on mixing. (More generally, by Jensen's inequality $\dimH = \E\log Z < \log \E Z$ unless $Z$ is constant.)

This straightforward description of harmonic measure for the NBRW allows one to directly control the location of this walk in the random graph. Consequently, by adding a few ingredients to the approach originally used in~\cite{LS10}, we were able to extend the NBRW analysis of that work to the non-regular setting (Theorems~\ref{T:mix_average}--\ref{T:mix_worst}; similar results for the NBRW were independently obtained in~\cite{Announce}; see~\S\ref{sec:NBRW} for further details).

However, the harmonic measure for the simple random walk (SRW) remains mysterious, and  there is no explicit formula for $\dimH$ even for very simple offspring distributions (see Fig.~\ref{fig:dim}).
Formally, let $T$ be an infinite GW-tree rooted at $\rho$ with offspring distribution $Z$. Under our assumptions the random walk $(X_t)$ is transient, so its
loop-erasure defines a unique ray $\xi$. Denoting graph distance by $\dist(\cdot,\cdot)$, let $\nu$ be the asymptotic speed of the walk (well-defined for almost every tree; see~\cite[\S3]{LPP95})) given by
\begin{equation}
  \label{eq-nu-def}
  \nu \stackrel{\text{a.s.}}=\lim_{t\to\infty} \tfrac1t \dist(X_t,\rho)\,.
\end{equation}
Let $\P_T$ denote the conditional probability given $T$, and set
\begin{equation}
  \label{eq-W-def}
W_T(v) = -\log\P_T(v\in\xi)\,.
\end{equation}
Consider the metric $d(\xi,\eta)=\exp(-|\xi\wedge\eta|)$
on all rays $\partial T$, where $\xi\wedge\eta$ is the longest common prefix of $\xi$ and $\eta$.
It was shown in~\cite{LPP95} that in the joint probability space of a GW-tree and SRW on it,
\begin{equation}
  \label{eq-W(Yt)-LPP}
  W_T(\xi_t)/ t \stackrel{\text{a.s.}}\longrightarrow \dimH\,, 
\end{equation}
where $\dimH$ is non-random and is the Hausdorff dimension of harmonic measure in the above metric.
Building on the results of~\cite{LPP95,LPP96,DGPZ02}, we establish the following refinement of~\eqref{eq-W(Yt)-LPP}, which plays a central role in our proof and seems to be of independent interest.
\begin{mainprop}
  \label{prop-variance}
Let $T$ be a Galton-Watson tree, conditioned to survive, whose offspring variable $Z$ satisfies
$1< \E Z < \infty$. For every $\epsilon>0$ there exists $\delta>0$ so that, for all $t$,
\begin{equation}
  \label{eq-WT-var-GW-SRW}
  \mbox{$\P\left( \left|W_T(\xi_t) - \dimH t \right| > \delta\sqrt{t}\right) < \epsilon$}\,.
  \end{equation}
\end{mainprop}
Indeed, the upper bound on $\tmix$ will hinge on showing that w.h.p.\ at a suitable time $t = (\spe\dimH)^{-1}\log n + O(\sqrt{\log n})$ the $L^2$-distance of SRW from equilibrium is $\exp(O(\sqrt{\log n}))$ --- a term that originates from the $O(\sqrt{t})$-fluctuations of $W_T(\xi_t)$ as per Proposition~\ref{prop-variance}.

\subsection*{Organization and notation}
In \S\ref{sec:GW} we will establish Proposition~\ref{prop-variance} along with several other estimates for random walk on GW-trees, building on the works of \cite{LPP95,LPP96,DGPZ02}.
Section~\ref{sec:SRW} studies SRW on random graphs and contains the proofs of Theorem~\ref{mainthm-giant} and~\ref{mainthm:deg-seq}, while Section~\ref{sec:NBRW} is devoted to the analysis of the NBRW.

Throughout the paper, a sequence of events $A_n$ is said to hold with high probability (w.h.p.) if $\P(A_n)\to1$ as $n\to\infty$. We use the notation  $f \ll g$ and $f \lesssim g$ to abbreviate $f = o(g)$ and $f=O(g)$, resp.\ (as well as their converse forms), and $f\asymp g$ to denote $f\lesssim g\lesssim f$. Finally, in the context of an offspring distribution $Z$ with $\E Z>1$, we say that $T$ is a $\gwsurv$-tree to refer to the corresponding GW-tree conditioned on survival.

\section{Random walk estimates on Galton-Watson trees}\label{sec:GW}

Let $T$ be an infinite tree, rooted at some vertex $\rho$, on which random walk is transient.
In what follows, we will always use $(X_t)$ to denote SRW on $T$ and let the random variable $\xi\in\partial T$ denote its limit, \emph{i.e.}, the unique ray that $(X_t)$ visits i.o., or equivalently, the loop-erased trace of $(X_t)_{t=0}^\infty$. For any vertex $v\in T$ other than the root, we let $v^{-}$ denote its parent in $T$, and let $\theta_T(v)$ denote the incoming flow at $v$ relative to its parent corresponding to  harmonic measure:
\begin{equation}\theta_T(v) = \P_T(v \in \xi \mid v^-\in\xi)\,,
  \label{eq-weight-of-v}
\end{equation}
so that if $(v_0=\rho,v_1,\ldots,v_k=v)$ is a shortest path in $T$ then $\P_T(v\in \xi) = \prod_{i=1}^k \theta_T(v_i)$.

Our goal in this section is to establish Proposition~\ref{prop-variance} as well as the next two estimates, in each of which the underlying offspring distribution $Z$ is assumed to have $\E Z>1$.
\begin{definition*}[Hitting measure]
  For a tree $T$ rooted at $\rho$ and an integer $R$, let $T_R$ denote the tree induced on $\{v : \dist(\rho,v)\leq R\}$. For $v\in T_R$, let
  $\tilde\theta_{T_R}(v)$ denote the probability that SRW from $\rho$ first hits level $R$ of $T$ at a descendent of $v$.
\end{definition*}
\begin{lemma}
  \label{lem-approx}
Let $T$ be a \gwsurv-tree rooted at $\rho$. There exists some $c>0$ (depending only on the law of $Z$) so that for every $R>0$ the following holds. With \gwsurv-probability at least $1-\exp(-cR)$, every child $v$ of $\rho$ satisfies
  \begin{align} \Big|\theta_T(v) - \tilde\theta_{T_R}(v)\Big| \leq \exp(-c R)\,,\label{eq-theta-approx} \\
  \Big|\log \theta_T(v) - \log \tilde\theta_{T_R}(v) \Big| \leq \exp(-c R) \,.\label{eq-log-theta-approx}
  \end{align}
\end{lemma}
\begin{lemma}
  \label{lem-detour}
Let $T$ be a \gwsurv-tree. There exists $c>0$ such that, for all $R,t>0$,
  \[ \P\left(\dist(X_t,\xi) > R\right) \leq \exp(-c R)\,.\]
\end{lemma}

\subsection{Proof of Lemma~\ref{lem-approx}}
Let $Z_1$ be the degree of $\rho$, and denote the children of $\rho$ as $v_1,v_2,\ldots,v_{Z_1}$. We will show that~\eqref{eq-theta-approx}--\eqref{eq-log-theta-approx} hold for $v=v_1$ except with probability $\exp(-cR)$, and the desired statement will follow from a union over the children of $\rho$ and averaging over $Z_1$ (multiplying said probability by the expectation of $Z_1$ w.r.t.\ \gwsurv, which is $O(1)$).

Let $L_i$ be the set of leaves of the depth-$(R-1)$ subtree of $v_i$ (i.e., descendants of $v_i$ at distance $(R-1)$ from it) and let $\tau = \min\{ t : X_t \in \cup L_i\}$ be the hitting time of SRW to level $R$ of $T$, so that $\tilde{\theta}_{T_R}(v) = \P_T( X_{\tau} \in L_1 )$.
Now let $B = \bigcup_{k>0}\big\{ X_{\tau+k} = \rho \big\}$ denote the event that $X_t$ revisits $\rho=v^-$ after time $\tau$. Clearly, on the event $B^c$, we have $\xi_1=v$ iff $X_{\tau}\in L_1$, and so
\begin{align}
  \label{eq-PTB}
  \left| \theta_T(v) - \tilde\theta_{T_R}(v) \right| \leq \P_T(B)\,.
\end{align}
Estimating $\P_T(B)$ follows from the following result in~\cite[p21]{DGPZ02}, which was obtained as a corollary of a powerful lemma of Grimmett and Kesten~\cite{GK} (cf.~\cite[Lemma~2.2]{DGPZ02}).
\begin{lemma}\label{lem:GK-DGPZ}
There exist $c,\alpha>0$ such that the following holds. Let $T$ be a \gwsurv-tree and let $R\geq1$. With probability at least $1-\exp(-cR)$, every depth-$R$ vertex $x$ satisfies that the ray $P$ from the root to $x$ has at least $\alpha R$ vertices $v$ such that $\P_v(\tau^+_{P}=\infty)>c$, where $\tau^+_{P}$ denotes the return time of RW to the ray $P$.
\end{lemma}
By this lemma, there is a measurable set $\cA$ of \gwsurv-trees such that $\P(\cA) \leq \exp(-cR)$ and for every $T\notin \cA$, every $z\in\cup L_i$ satisfies the following: there are at least $\alpha R$ vertices on its path to $\rho$ so that the SRW from such a vertex $y$ has a probability $q>0$ of escaping to $\infty$ never again visiting its parent $y^-$. In particular,
\begin{equation*}
   \P_T(B)\one_{T\notin\cA} \leq (1-q)^{\alpha R} \leq \exp(-c' R)
\end{equation*}
for some $c'>0$, which we take to be at most $c/2$. This establishes~\eqref{eq-theta-approx}.

It remains to show~\eqref{eq-log-theta-approx}. Let $\gamma_T$ be the probability that SRW on $T$ does not revisit the root beyond time 0,
and notice that by definition, for every $i$,
\[ \theta_T(v_i) \wedge \tilde\theta_{T_R}(v_i) \geq \gamma_i / Z_1\,,\]
where $\gamma_i$ is the copy of $\gamma_T$ corresponding to the subtree rooted at $v_i$. By~\cite[Lemma~9.1]{LPP95},
\[ \E[1/ \gamma] \leq \frac{1}{1-\E[1/Z]} =O(1)\,.\]
Define the event
\[ \tilde\cA = \cA \cup \big\{\gamma_1 < e^{-\frac12c'R}\big\}\cup\big\{Z_1 > e^{\frac14 c'R}\big\}\,, \]
and observe that $\P(\tilde\cA) = O(e^{-\frac14 c'R})$ by Markov's inequality.

On the event $\tilde\cA^c$,
\begin{equation*}
     \left|\log \theta_T(v) - \log\tilde\theta_{T_R}(v)\right| \leq \frac{\big|\theta_T(v)-\tilde\theta_{T_R}(v)\big|}{\theta_T(v) \wedge \tilde\theta_{T_R}(v)} \leq \frac{\P_T(B)}{\gamma_1/Z_1} \leq \frac{e^{-c'R}}{e^{-\frac34 c'R}} = e^{-\frac14 c' R}\,,
 \end{equation*}
 establishing~\eqref{eq-log-theta-approx}.
\qed

\subsection{Proof of Lemma~\ref{lem-detour}}
Following~\cite{LPP95}, consider the Augmented GW-tree (AGW), which is obtained by joining the roots of two i.i.d.\ GW-trees by an edge. A highly useful observation of~\cite{LPP95} is that the process in which SRW acts on $(T,\rho)$ by moving the root to one of its neighbors is stationary when started at an AGW-tree $T$ and one of its roots. Clearly, the probability of the event under consideration here in the GW-tree is, up to constant factors from below and above, the same as the analogous probability in the AGW-tree (with positive probability the walk never traverses the edge joining the two copies; conversely, a union bound can be applied to the two GW-tree instances). For brevity, an $\agwsurv$-tree will denote an AGW-tree conditioned to be infinite, for which the above stationarity property w.r.t.\ SRW is still maintained.

The event $\dist(X_t,\xi)> R$ implies that, if $u$ is the vertex at distance $R$ from $X_t$ on the loop-erased trace of $(X_0,X_1,\ldots,X_t)$ (the result of  progressively erasing cycles as they appear) then the SRW must revisit $u$ at some time $s>t$.

The stationarity reduces our goal to showing that $\P(\bigcup_{s>0} \{X_s = u\})\leq \exp(-cR)$, where $u$ is the vertex at distance $R$ from $X_0$ on the loop-erased trace of $(X_0,\ldots,X_{-t})$.
Condition on the past of the walk.
By Lemma~\ref{lem:GK-DGPZ}, with probability $1-O(\exp(-cR))$ the \agwsurv-tree $T$ satisfies that the path from $\rho$ to $u$ contains some $\alpha R$ vertices from which the walk would escape and never return to this path with probability bounded away from 0, whence the probability of visiting $u$ at a positive time is at most $\exp(-cR)$ for some $c>0$.
\qed

\subsection{Proof of Proposition~\ref{prop-variance}}\label{sec:proof-var}
A \emph{regeneration point} for the SRW on the \gwsurv-tree is a time $\tau\geq 1$ such that the edge $(X_{\tau-1},X_\tau)$ is traversed exactly once along the trace of the random walk.
Set $\tau_0=0$ and let $\tau_1<\tau_2<\ldots$ denote the sequence of regeneration points on a random \gwsurv-tree $T$. The following remarkable result due to Kesten was reproduced in~\cite{Piau96} (see also~\cite{LPP95,LPP96} where the first part was observed).
\begin{lemma}\label{lem:regen}
Let $T$ be a \gwsurv-tree rooted at $\rho$, let $(\tau_i)_{i\geq 0}$ be the regeneration points of SRW on it, and let $\varphi_i=\dist(\rho,X_{\tau_i})$ denote their depths in the tree.
\begin{enumerate}[(a)]
  \item Let $T_i$ ($i\geq 1$) denote the depth-$(\varphi_i-\varphi_{i-1})$ tree that is rooted at $X_{\tau_{i-1}}$.
Then $\{(T_i,(X_t)_{\tau_{i-1}\leq t < \tau_{i}})\}_{i\geq 1}$ are mutually independent, and for $i\geq 2$ they are i.i.d.
\item There exist some $\alpha>0$ such that $\E \exp[\alpha (\varphi_{i}-\varphi_{i-1})] < \infty$.
\end{enumerate}
\end{lemma}
By the standard decomposition of a \gwsurv-tree into a GW-tree with no leaves (the \emph{skeleton} of the tree) and critical bushes (see, e.g.,~\cite[{\S}I.12]{AN72}), together with the fact that the harmonic measure is supported on the skeleton, we may assume throughout this proof that $\P(Z = 0) = 0$.

Let $\varphi_i=\dist(\rho,X_{\tau_i})$ ($i\geq 1$) as in the above lemma (noting that necessarily $X_{\tau_i} = \xi_{\varphi_i}$ as regeneration points are by definition part of the loop-erased trace $\xi$) and set
\[ Y_i = -\log\theta_T(\xi_i)\,,\]
so that our goal is to show the concentration of  $\sum_{i=1}^t Y_i = W_T(\xi_t)$. Next we show that
$Y_i$ has exponential moments. If $v_1,\ldots,v_{Z_1}$ are the children of the root (recall $Z_1>0$) then
$\theta_T^{-1}(\xi_1)=\sum_{i=1}^{Z_1}\one_{\xi_1=v_i} \theta_T^{-1}(v_i)$ with the summands being i.i.d.\ given $Z_1$. Hence,
\begin{align*}
   \E\left[\theta_T^{-1}(\xi_1) \mid Z_1\right]&= Z_1 \E\left[\one_{\xi=v_1} \theta_T^{-1}(v_1)\mid Z_1\right] =Z_1 \E\left[\theta_T^{-1}(v_1) \E\left[\one_{\xi=v_1} \mid T\right]\mid Z_1\right]=  Z_1   \,,
 \end{align*}
and it follows that
\[ \E\exp(Y_1) = \E \left[ \theta_T(\xi_1)^{-1}\right] = \E Z_1 = O(
1)\,.\]

By~\cite[Theorem~8.1]{LPP95} there exists a probability measure \gwharm\ on GW-trees which satisfies the following. It is stationary for the harmonic flow $\theta_T$, mutually absolutely continuous w.r.t.\ GW (moreover, $\gwharm / \mathrm{GW}$ is uniformly bounded from above), and under $\gwharm$ one has $\E[W_T(\xi_t)] = \dimH t$. We will show that
\begin{equation}
  \label{eq-var-WT-muharm}
  \var(W_T(\xi_t)) \leq C_z t\quad\mbox{ w.r.t.\ \gwharm$\ltimes$SRW}
\end{equation}
for some constant $C_z$ that depends only on the law of the offspring variable $Z$, where $\gwharm\ltimes $SRW is the joint distribution of a \gwharm-tree and SRW on it.
 A direct application of Chebyshev's inequality will then imply~\eqref{eq-WT-var-GW-SRW} under \gwharm$\ltimes$SRW (for $\delta = \sqrt{C_z / \epsilon}$). The analogous statement for GW$\ltimes$SRW will then follow by the absolute continuity: for every $\epsilon>0$ there exists $\epsilon'$ such that if the event in~\eqref{eq-WT-var-GW-SRW} has probability at most $\epsilon'$ under \gwharm$\times$SRW then it has probability at most $\epsilon$ under GW$\ltimes$SRW.

It remains to show~\eqref{eq-var-WT-muharm}. Note that the bounds established above on  $\E Y_1^2$ under GW remain valid under \gwharm\ thanks to the fact that $d\gwharm/d$GW$=O(1)$.

Since the sequence $(Y_i)_{i\geq 1}$ is stationary under $\gwharm\ltimes$SRW, it suffices to show that there exist some constants $c,c'>0$ such that $\cov(Y_1,Y_k) \leq c\exp(-c' k)$ for all $k$. Let $T_{\lfloor k/2\rfloor }$ be the depth-$\lfloor k/2\rfloor$ subtree of $T$ (sharing the same root), and write $\theta(v) = \theta_T(v)$ and $\tilde\theta(v) = \tilde\theta_{T_{\lfloor k/2\rfloor}}(v)$ for brevity.
Denoting the children of the root by $v_1,\ldots,v_{Z_1}$, recall from Lemma~\ref{lem-approx} (with $R=\lfloor k/2\rfloor$) and Lemma~\ref{lem:regen} that there exist some $c_1,c_2>0$ (depending only on the offspring distribution $Z$) such that the events $\cE_1,\cE_2$ given by
\[ \cE_1 = \bigcup_{1\leq i\leq Z_1}\left\{ |\tilde\theta(v_i)-\theta(v_i) | > e^{-c_1 k}\right\}\,,\qquad \cE_2= \{\varphi_1 \geq k/2\}\]
satisfy $\P(\cE_1 \cup \cE_2) \leq \exp(-c_2 k)$.
It follows from H\"older's inequality that
\[ \cov\left(Y_1 \,\one_{\cE_1},Y_k\right) \leq
\big[\E Y_1^2\big]^{\frac12} \big[\E Y_k^4\big]^{\frac14} \P(\cE_1)^{\frac14}\lesssim \exp(-c_2 k)\]
(recall $Y_i$ has finite moments of every order), and similarly for
$\cov\left(Y_1\one_{\cE_1^c} , Y_k \one_{\cE_2}\right)$; hence, to estimate $\cov(Y_1,Y_k)$ it remains to bound $\cov(Y_1 \one_{\cE_1^c},Y_k\,\one_{\cE_2^c})$.

On the event $\cE_1^c$, either (i) we have $\theta(v_i) \leq \exp(-\frac34 c_1 k)$, whence
\[ \theta(v_i)\log^2\theta(v_i) \lesssim e^{-\frac12 c_1 k}\quad\mbox{ and }\quad
\tilde\theta(v_i)\log^2\tilde\theta(v_i) \lesssim e^{-\frac12 c_1 k}\,,\]
 or (ii) we have $\theta(v_i)> \exp(-\frac34 c_1 k)$ and infer from the bound on $\left|\tilde\theta(v_i)-\theta(v_i)\right|$ that
  \[ \left| \log \theta(v_i)-\log \tilde\theta(v_i) \right|^2 \leq \bigg(\frac{e^{-c_1 k}}{\theta(v_i)\wedge \tilde\theta(v_i)}\bigg)^2 \lesssim  e^{-\frac12 c_1 k}\,,
\]
 where in both cases the implicit constants are independent of $k$.
Hence, if we let $Y'_1$ assume the value $\log(1/\tilde\theta(v_i))$ with probability $\tilde\theta(v_i)$ for each child $v_i$ of the root of $T$,
\[
\E \left| Y_1 \one_{\cE_1^c} - Y'_1\right|^2  \lesssim e^{-\frac12 c_1 k} \E Z \lesssim e^{-\frac12 c_1 k}\,;
\]
thus (recalling that $\E Y_k^2 = O(1)$),
\[ \cov(Y_1\one_{\cE_1^c} -Y'_1,\, Y_k \one_{\cE_2^c}) \leq \sqrt{\E (Y_1 \one_{\cE_1^c} -Y'_1)^2 \,\E Y_k^2} \lesssim e^{-\frac14 c_1 k}\,.
\]
Thanks to the discrete renewal theorem---crucially using that the renewal intervals $\varphi_i$ have exponential moments (see~\cite{Kendall59} as well as~\cite[{\S}II.4]{Lindvall92})---we can couple $Y_k\one_{\cE_2^c}$ with an i.i.d.\ copy of $Y_1$ with probability $1-O(\exp(-c_3 k))$, and so $\cov(Y'_1, Y_k\one_{\cE_2^c}) \lesssim \exp(-c_3 k)$.
This completes the desired bound on $\cov(Y_1,Y_k)$ and concludes the proof.
\qed

\section{Random walk from a typical vertex in a random graph}\label{sec:SRW}
We begin by studying SRW and LERW on an infinite GW-tree along what would be the cutoff window. This analysis will then carry over to SRW on $G$ via coupling arguments, and provide initial control over the distribution of the walk within the cutoff window. The final step will be to boost mixing in that window via the graph spectrum.

\subsection{Quantitative estimates on the infinite tree}\label{sec:tree-estimates}
Fix $\epsilon>0$ and set
\begin{align}
  m &= n \exp\left(-\gamma\sqrt{\log n}\right)\,, \label{eq-m-def}\\
  \ell_0 &= \frac1{\dimH} \log m = \frac1{\dimH}\log n - O(\sqrt{\log n})\,,\label{eq-t0-def}\\
  \ell_1 &= \ell_0 + \frac{1}{4\dimH}\gamma\sqrt{\log n}\,,\label{eq-t1-def}
\end{align}
where $\gamma>0$ is a suitably large constant with the following two properties:
\begin{enumerate}[(i)]
  \item
By Proposition~\ref{prop-variance}, for large enough $\gamma$ (in terms of the law of $Z$ and $\epsilon$),
\begin{align}
\P\left(  W_T(\xi_{\ell_0}) > \log m - \tfrac13\gamma\sqrt{\log n}\right) 
   > 1-\epsilon\,,\label{eq-W(t0)-epsilon-bnd}\\
   \P\left(W_T(\xi_{\ell_1}) < \log m + \tfrac13\gamma\sqrt{\log n}\right) 
   > 1-\epsilon\,.\label{eq-W(t1)-epsilon-bnd}
 \end{align}
\item It is known (see~\cite{Piau96}) that the distance of SRW $(X_t)$ from the root (its origin) of GW-a.e.\ tree has mean $\spe t+O(1)$ and variance at most $C_z t$ for some fixed $C_z>0$ (which depends only on the offspring variable $Z$).
Thus,
\[\P\left(\left|\dist(\rho, X_t) - \spe t\right|\leq\tfrac{1}{16\dimH}\gamma\sqrt{\log n}\right) >1-\epsilon\quad \mbox{ for any $t\leq \spe^{-1}\ell_1$}
\]
provided $\gamma$ is large enough in terms of $\dimH$, $C_z$ and $\epsilon$.
Combined with the exponential tails of the regeneration times (recall~\S\ref{sec:proof-var}), taking $\gamma$ large enough further gives
\begin{align}
  \label{eq-speed-epsilon-bnd}
  \P\left(\max_{t \leq \spe^{-1}\ell_1}\left|\dist(\rho, X_t) - \spe t\right|\leq\tfrac{1}{16\dimH}\gamma\sqrt{\log n}\right) >1-\epsilon\,.
\end{align}

\end{enumerate}
Let $c_0>0$ be some constant (to be specified later), and set
\begin{equation}
  \label{eq-R-def}
  R = c_0 \log \log n \,.
\end{equation}
\begin{definition*}
For a tree $T$ and integer $R>0$, if $u_0=\rho,u_1,\ldots,u_k=u$ is the shortest path from the root to a vertex $u$ and $\tilde\theta$ is as defined in \S\ref{sec:GW},
define
\[ \widetilde{W}_T(u) = -\sum_{i=1}^{k} \log\tilde\theta_{T_R(u_i^-)}(u_i)\,,\]
where $T_R(u_i^-)$ is the depth-$R$ subtree of $T$ rooted at the parent of $u_i$.
\end{definition*}
Consider the exploration process where at step $k$ --- corresponding to $\cF_k$ in the filtration --- we expose the depth-$(k+1)$ descendants of the root, as in the standard breadth-first-search process, with one modification: we explore the children of a vertex $v$ (which is at depth $k$) only if its distance-$R$ ancestor, denoted $u$, satisfies
\begin{equation}\label{eq-bad-u}
 \widetilde{W}_T(u) \leq \log m +  \tfrac12 \gamma\sqrt{\log n} \,.
\end{equation}
Instead, if~\eqref{eq-bad-u} is violated, we declare that the exploration process is \emph{truncated} at all the depth-$R$ descendants of $u$, and denote this $\cF_k$-measurable event by $\btr(u)$.

We stress that, when determining whether to truncate the exploration at the vertex $v$ as described above, the variable $\widetilde{W}_T(u)$ is $\cF_k$-measurable (as the event $\btr(u_i)$ has not occurred for any of the ancestors of $u$, thus the subtree of depth $R$ of each of them has been fully explored).
\begin{lemma}\label{lem-truncation}
Let $T$ be a \gwsurv-tree, and define $\ell_1$, $\gamma$ and $\btr$ as above. If $c_0$ from~\eqref{eq-R-def} is sufficiently large, then
  \[ \P\Big(\bigcup_{k \leq \ell_1}\btr(\xi_k)\Big) \leq \P\left(W_T(\xi_{\ell_1})> \log m + \tfrac13 \gamma\sqrt{\log n}\right) + o(1) <\epsilon+o(1)\,.\]
\end{lemma}
To prove this result, we need the following simple claim.
\begin{claim}\label{clm-lerw-bad-trees}
Let $T$ be a \gwsurv-tree.
There exist  $c_1,c_2>0$ such that the following holds. For every measurable set $\cA$ of trees and integer $\ell$,
\[ \P\Big(\bigcup_{i\leq \ell}\left\{T_\infty(\xi_i)\in \cA\right\}\Big) \leq c_1  \ell \P\left(\cA\right) + \exp\left(-c_2 \ell^{1/3}\right)\,.\]
\end{claim}
\begin{proof}
It suffices to prove this for the \agwsurv-tree (as argued in the proof of Lemma~\ref{lem-detour}).
Let $(X_t)_{t=0}^\infty$ be simple random walk on $T$ whose loop-erased trace is $\xi$. Then
\begin{align*} \P\Big(\bigcup_{i\leq\ell}T_\infty(\xi_i) \in \cA\Big) &\leq \P\Big( \bigcup_{t \leq 2\ell / \nu} \left\{ T_\infty(X_t) \in \cA\right\}\Big) + \P\Big(\bigcup_{t\geq 2\ell/\nu} \left\{\dist(X_t,\rho) < \ell\right\}\Big)\\
&\leq \frac{2\ell}{\nu} \P(\cA) + \sum_{t \geq 2\ell/\nu} \exp\left(-c_3 t^{1/3}\right)\,,
\end{align*}
using the stationarity of the \agwsurv-tree and the large deviation estimate from~\cite{Piau96} on the speed of random walk on a GW-tree (see also~\cite{DGPZ02}).
\end{proof}
\begin{proof}[\emph{\textbf{Proof of Lemma~\ref{lem-truncation}}}]
Note that the right-hand inequality follows from the choice of $\gamma$ as per~\eqref{eq-W(t1)-epsilon-bnd}, and it remains to prove the left-hand inequality.

Let $\cA$ denote the set of trees $T$ for which there exists a child $v$ of the root such that
\[ \left| \log\tilde\theta_{T_R}(v)- \log\theta_T(v) \right | > \exp(-cR)\]
for $c>0$ the constant from Lemma~\ref{lem-approx}.
Note that $\P(\cA)  < \exp(-cR)$ by~\eqref{eq-log-theta-approx} in that lemma. Furthermore, on the event
\[ \cE_k := \bigcap_{i=0}^{k-1} \Big( \btr(\xi_i)^c  \cap \left\{ T_\infty(\xi_{i})\notin \cA\right\}\Big) \cap \btr(\xi_k)
\]
we have
\[
\left|\log \tilde\theta_{T_R}(\xi_i)- \log\theta_T(\xi_i)\right|\leq \exp(-c R)\qquad\mbox{ for all $i=1,\ldots,k$.}
\]
Thus, by the triangle inequality,   the event $\cE_k$ implies that
\[
\left| \widetilde{W}_T(\xi_k) - W_T(\xi_k)\right| \leq k e^{-cR} \lesssim e^{-cR} \log n =O\big(\log^{-2} n\big)\,,
 \]
provided that $k=O(\log n)$ and $c_0 \geq 3/c$ in the definition~\eqref{eq-R-def} of $R$.
Therefore, by the truncation criterion~\eqref{eq-bad-u}, we find that $\cE_k$ implies that
\[
 W_T(\xi_k) \geq \log m + \tfrac12 \gamma\sqrt{\log n} - O\big(\log^{-2} n\big)> \log m + \tfrac13 \gamma\sqrt{\log n}
 \]
for sufficiently large $n$, whence (using that $W_T(\xi_k) \leq W_T(\xi_{\ell_1})$ for all $k\leq \ell_1$),
\begin{equation}\label{eq-W(u)-approx}
\bigcup_{k\leq \ell_1}\cE_k \subset \Big\{ W_T(\xi_{\ell_1}) \geq \log m + \tfrac13 \gamma\sqrt{\log n}\Big\}\,.
\end{equation}
On the other hand, Claim~\ref{clm-lerw-bad-trees} shows that  
\begin{equation}
\label{eq-union-bad-A-bnd} \P\Big(\bigcup_{i\leq \ell_1}\left\{T_\infty(\xi_i)\in\cA\right\}\Big) \leq c_1 \ell_1 e^{-cR} + e^{-c_2\ell_1^{1/3}}
= O\big(\log^{-2} n\big)\,.
\end{equation}
Since, by the definition of $\cE_k$,
\[ \bigcup_{k\leq \ell_1} \btr(\xi_k) \subset \Big( \bigcup_{k\leq\ell_1} \cE_k \Big) \cup \Big( \bigcup_{k\leq\ell_1}\left\{T_\infty(\xi_i)\in\cA\right\}\Big) \,,\]
combining~\eqref{eq-W(u)-approx}--\eqref{eq-union-bad-A-bnd} completes the proof.
\end{proof}
Let $S_k$ ($k\leq \ell_1$) denote the set of explored vertices at distance $k$ from the root (so that $S_k$ is $\cF_k$-measurable), and define the $\cF_{k+R}$-measurable $S'_k$ by
\[ S'_k = \left\{ u \in S_k \;:\; \btr(u)^c\right\}\,.\]
By the definition of the truncation criterion, for every $u\in S'_k$,
\[ -\widetilde{W}_T(u) \geq -\log m - \tfrac12\gamma\sqrt{\log n}\,.\]
It is easy to see (e.g., by induction on $k$) that
\[ \sum_{u\in S'_k} \exp(-\widetilde W_T(u)) \leq 1.\]
In particular,
\[ 1\geq |S'_k| m^{-1} e^{- \frac12\gamma\sqrt{\log n}}\,,\]
thus it follows from~\eqref{eq-m-def} that
\[ |S'_k| \leq m \exp\left(\tfrac12 \gamma\sqrt{\log n}\right) = n \exp\left(-\tfrac12 \gamma\sqrt{\log n}\right)\,.\]
Therefore, as the maximum degree is at most $\Delta = \exp[(\log n)^{1/2-\delta}]$, the number of vertices reached in the first $\ell_1$ exploration rounds (i.e., those revealed in $\cF_{\ell_1}$) is at most
\begin{equation}
   \label{eq-explored-v}
   \sum_{k\leq \ell_1} |S_k| \leq \sum_{k\leq \ell_1} |S'_k|\Delta^R \leq \ell_1 n \exp\left(-\tfrac12\gamma\sqrt{\log n}\right) \Delta^{R} \leq n \exp\left(-\tfrac13\gamma\sqrt{\log n}\right)\,,
 \end{equation}
using the fact $\Delta^R = \exp[O((\log n)^{1/2-\delta} \log\log n)] = \exp[o(\sqrt{\log n})]$. (We stress that the estimate~\eqref{eq-explored-v}  holds with probability 1).

Finally, recall from~\eqref{eq-W(t0)-epsilon-bnd} that $W_T(\xi_{\ell_0}) > \log m - \frac13 \gamma\sqrt{\log n}$ with probability at least $1-\epsilon$.
Since Lemma~\ref{lem-truncation} implies that $\P(\xi_{\ell_0}\in S'_{\ell_0}) \geq 1-\epsilon-o(1)$,
it therefore follows (once we recall from~\eqref{eq-m-def} that $\frac{1}m \exp\left(\tfrac13 \gamma\sqrt{\log n}\right) = \frac1n \exp\left(\tfrac43\gamma\sqrt{\log n}\right)$) that
\[ \P(\xi_{\ell_0}\in S''_{\ell_0}) \geq 1-2\epsilon-o(1)\]
for $S''_{\ell_0}\subset S'_{\ell_0}$ defined by
\begin{equation}
   \label{eq-xi-t-upper}
S''_{\ell_0} = \Big\{ u\in S_{\ell_0}' \;:\;
   \P_T(\xi_{\ell_0} = u) \leq  \tfrac1n \exp\left(\tfrac43\gamma\sqrt{\log n}\right)\Big\}\,.
 \end{equation}

\subsection{Coupling the GW-tree to the random graph}\label{sec:coupling-to-GW}
We describe a coupling of SRW $(X_t)$ on a GW-tree\footnote{More precisely, the degree of the root will be different---given by $(p_k)$ as opposed to the size-biased law of $Z$---but this has no effect since this model is clearly contiguous to the standard GW-tree.}
 $T$, started at its root $\rho$,
to SRW $(\sX_t)$ on $G$, started at a fixed vertex $v_1$. Let $\tilde{T}\subset T$ be the truncated GW-tree from~\S\ref{sec:tree-estimates}, and let $\tau$ be the time at which the SRW $(X_t)$ on $T$ reaches level $\ell_1$.
 We will first construct a coupling of $G$ and $T$, in which we will specify a tree $\tilde{\Gamma}\subset G$ rooted at $v_1$ and a one-to-one mapping  $\phi:\tilde{\Gamma} \to \tilde{T}$, such that
 \begin{equation}
   \label{eq-phi-properties}
   \deg_G(x) = \deg_{T}(\phi(x))\quad \mbox{for all $x\in\tilde\Gamma$}\,,
 \end{equation}
(where $\deg_G(u)$ accounts for multiple edges). Moreover, defining, for $k\geq 1$, the event
 \[ \couple_k := \bigcap_{t\leq k}\{X_t \in \phi(\tilde\Gamma)\}\,,  \]
 we will show that
 \begin{equation}
   \label{eq-successful-coupling}
   \P(\couple_\tau)  \geq 1 - \epsilon - o(1)\,.
 \end{equation}
In what follows, we refer to the event $\Pi_k$ as a ``successful coupling'' of $(\sX_t,X_t)_{t=1}^k$, for the following reason:
denoting $\tilde\tau = \min\{ t : X_t\notin \phi(\tilde\Gamma)\}$, take $\sX_t=\phi^{-1}(X_t)$ for $t < \tilde\tau$,
let $\sX_{\tilde\tau}$ choose a uniform neighbor of $\sX_{\tilde\tau-1}$ among those that belong to $G \setminus \tilde\Gamma$ (recalling that $\deg_G(\sX_{\tilde\tau-1}) = \deg_T(X_{\tilde\tau-1})$) and run $\sX_t$ independently of $X_t$ for $t>\tilde\tau$. This yields the correct marginal of SRW on $G$ by~\eqref{eq-phi-properties}. When using this coupling, on the event $\Pi_\tau$ one has $\sX_t= \phi^{-1}(X_t)$ for all $t\leq \tau$.

Construct the coupling of $(G,T)$ by exposing, in tandem, the truncated GW-tree $\tilde T$ and a subgraph of the exploration process of $v_1$ in $G$ via the configuration model.\footnote{Every vertex $u$ is associated with a number of ``half-edges'' corresponding to its degree, and the (multi)graph is generated by a uniform perfect matching of the half-edges (see, e.g.,~\cite{Bollobas-RG}). In our context, we reveal the connected component of $v_1$ by successively revealing the paired half-edges of vertices in a queue (breadth-first-search).  If multiple edges or loops are found, or the component is exhausted with at most $\sqrt{n}$ vertices, we abort the analysis (rejection-sampling for having $v_1\in\cC_1(G)$).} Initially, we let $\deg_G(v_1) = \deg_T(\rho)$ and set $\tilde\Gamma=\{v_1\}$ and $\phi(v_1)=\rho$. We proceed by exposing, vertex by vertex, the tree $\tilde{T}$, and attempting to find $\tilde\Gamma$ in $G$ such that $\tilde\Gamma$ is isomorphic to a subtree of $\tilde T$ rooted at $\rho$, while maintaining that the queue of active vertices (ones that await exploration) in $G$ will always be a subset of  $\tilde\Gamma$.

Suppose we are now exploring the offspring of $z\in \phi(\tilde\Gamma)$. Note one can map the offspring $y_1,\ldots,y_r$ of $z$ one-to-one to the half-edges of $u=\phi^{-1}(z)$ in $G$, a currently active vertex in the exploration queue, with~\eqref{eq-phi-properties} in mind.
\begin{enumerate}[(a)]
  \item  If the offspring of $z$ are to be truncated, remove $u$ from the exploration queue of $G$ (do not explore its half-edges).
  Otherwise, process $y_1,\ldots, y_r$ sequentially as follows.
  \item If the half-edge corresponding to $y_i$ is matched to some previously encountered vertex $v$ in the exploration of $G$ (i.e., the edge $(u,v)$ formed by this pairing closes a cycle in $G$). In this case, mark both $u$ and $v$ by saying that the events $\bcyc(u)$ and $\bcyc(v)$ hold, and remove $u$ and its offspring from the exploration queue of $G$. In addition, we delete $u$ and its offspring from $\tilde\Gamma$.

\item  If, on the other hand, the half-edge corresponding to $y_i$  is to be matched to a vertex $v$ that has not yet been encountered, we sample $\deg_T(y_i)$ and $\deg_G(v)$ via the maximal coupling between their distributions (the former has the law of $Z$, and the law of the latter is a function of the remaining vertex degrees in $G$). If this sample has $\deg_T(y_i)=\deg_G(v)$, then we add $v$ to $\tilde\Gamma$ and set $\phi(v)=y_i$; otherwise, we say $\bdeg(v)$ holds, and remove $v$ from the exploration queue.
\end{enumerate}
The above construction satisfies~\eqref{eq-phi-properties} by definition. To verify~\eqref{eq-successful-coupling}, we must estimate the probability that one of the following events occurs: (a)~the walk $(X_t)$ visits a truncated vertex (i.e., $\btr(X_t)$ occurs for some $t\leq \tau$), (b)~its counterpart $(\sX_t)$ visits a vertex with non-tree-edges (i.e., $\bcyc(\sX_t)$ occurs for some $t\leq \tau$), or (c)~the walk $(X_t)$ visits some vertex $y\in\tilde T$ such that its counterpart in $G$ had a mismatched degree (that is, $\bdeg(\sX_t)$ occurs for some $t\leq \tau$).

 The estimate of the event in~(a) will follow directly from our results in~\S\ref{sec:tree-estimates}. For the sake of estimating the events in~(b)--(c), and in order to circumvent potentially problematic dependencies between the random walk and the random environment, we will estimate the following larger events: if $w^R$ denotes the distance-$R$ ancestor of a vertex $w$ in a tree, then let
  \[ \widehat{\cB}(w) = \bigcup\left\{\bcyc(v)\cup\bdeg(v) \;:\, v\in \tilde\Gamma_{R+1}(w^{R})\right\}\,,\]
that is, we test whether there exists a vertex with non-tree-edges or a mismatched degree in the  entire depth-$R$ subtree of the distance-$(R-1)$ ancestor of $w$. We will estimate the probability of $\widehat\cB = \bigcup_{i=0}^{\ell_1}\widehat{\cB}(\sX_{\varsigma_i})$, where $\varsigma_i$ is the first time that $\dist(\sX_t,v_1)=R+i$.

We first reveal the first $R+1$ levels of $\tilde\Gamma$ and $\tilde T$, and examine whether $\bcyc(v)$ or $\bdeg(v)$ hold for any one of their vertices. (If so, then $\widehat\cB$ occurs due to $i=0$.) For each $i\geq 1$, we reveal level $R+i+1$ of $\tilde\Gamma$ and $\tilde T$ in the following manner: we first run $\sX_t$ until $\varsigma_i$, then examine whether $\bcyc(v)$ or $\bdeg(v)$ hold for some $v$ that is a depth-$(R+1)$ descendent of $\sX^R_{\varsigma_i}$ (if so, $\widehat\cB$ occurs) and proceed to reveal the rest of this level.

Using this analysis, we guarantee that we examine whether $\bcyc(v)\cup\bdeg(v)$ occur for vertices $v$ in a given level, only at the first time that the walk has reached that level (thus preserving the underlying independence in the configuration model).
In addition, if $\sX_t$ leaves $\tilde\Gamma$ for the first time on account of either $\bcyc(\sX_t)$ or $\bdeg(\sX_t)$, then either $\widehat\cB$ holds, or $\sX_t$ must have revisited some vertex $u$ after reaching one of its depth-$R$ descendants.  The latter occurs by time $t$ with probability at most $t\exp(-cR)$ by Lemma~\ref{lem-detour}, which is $o(1)$ for $t\asymp \log n$ provided that the constant $c_0$ in the definition of $R$ is sufficiently large.

\begin{enumerate}
  [(a)]
  \item \emph{Avoiding a truncation:} By definition, a truncated vertex $v$ corresponds to a distance-$R$ ancestor $u$ for which the event $\btr(u)$ holds, and by Lemma~\ref{lem-truncation} the probability that the LERW $\xi$ visits such a vertex $u$ by time $\ell_1$ is at most $\epsilon+o(1)$.

      In the event that the LERW $\xi$ does not visit such a vertex, the SRW $(X_t)$ must escape from the ray $\xi$ to distance at least $R$ in order to hit a truncated vertex, a scenario which has probability
     at most $\exp(-cR)$ by Lemma~\ref{lem-detour}; thus, the overall probability of the SRW hitting a truncation by some time $t\asymp \log n$ is at most
      \[ t e^{-c R} \lesssim e^{-c R} \log n = o(1)\]
      provided the constant $c_0$ in the definition of $R$ is chosen large enough.

  \item \emph{Avoiding a cycle:} A set $S$ of size $s$ vertices in the boundary of the exploration process of $G$ has at most $\Delta s$ half-edges that are to be matched in the next round. A non-tree-edge at a given vertex $v\in S$ is generated by either
  \begin{enumerate}[1.]
  \item matching a pair of its half-edges together --- which has probability $O(\Delta^2/n)$ as long as the number of remaining unmatched half-edges has order $n$, or
  \item pairing one of them to another unmatched half-edge of $S$ (the more common scenario inducing cycles)  which has probability $O(s \Delta^2 /n)$.
  \end{enumerate}
      By~\eqref{eq-explored-v} we know that $s \leq n \exp(-\frac13 \gamma\sqrt{\log n}) =o(n/\Delta)$ and hence, indeed, the number of unmatched half-edges that remain at time $\ell_1$ is of order $n$.
      Thus, to bound the contribution to $\P(\widehat\cB)$ due to some $\bcyc(v)$, we estimate
      the probability of encountering such a vertex along some $\ell_1 \Delta^{R+1} $ trials. A union bound yields that this contribution has order at most
      \[ \ell_1\Delta^{R+1}\frac{s  \Delta^2 }{n} \lesssim \exp\left(-\tfrac13 \gamma\sqrt{\log n} + o\big(\sqrt{\log n}\big)\right)  = o(1)\,.\]

  \item \emph{Coupling vertex degrees:}
    While the tree uses a fixed offspring distribution $Z$, the exploration process in $G$ dynamically changes as the half-edges are sampled without replacement.
    Let $N_k$ denote the number of vertices of degree $k$ in $G$, and note that
    \[ \P\Big( \left| N_k - np_k \right| > 4 \left(\sqrt{n p_k \log n}\vee \log n\right) \Big) = O\left(n^{-3/2}\right) \]
    by a standard large deviation estimate for the binomial variable $N_k\sim\Bin(n,p_k)$,
    e.g.,~\cite[Theorem~2.1]{JLR-RG}, where the $O(\cdot)$ includes an extra factor of 2 due to the conditioning that the sum of the degrees, $\sum k N_k$, should be even. We may thus condition
    on the degrees $(D_u)_{u\in V}$ (whence $G$ becomes a random graph with a given degree sequence) and assume, by neglecting an event of probability $o(1)$,
    that
    \begin{align}\label{eq-Nk-1}
       |N_k - n p_k| \leq 4(\sqrt{n p_k\log n}\vee\log n)\quad\mbox{ for all $k=1,\ldots,\Delta$}
     \end{align}and that $N_k=0$ for $k>\Delta$.
   Further observe that, by Cauchy-Schwarz,
    \begin{align}\label{eq-Nk-2} \sum_{k=1}^\Delta (\sqrt{np_k\log n} \vee \log n) \leq \sqrt{\Delta n\log n} \vee \Delta \log n \leq n^{1/2+o(1)}\,.\end{align}

    Letting $\tilde{N}_k$ denote the number of vertices with $k$ unmatched half-edges at some point in time, a given half-edge gets matched to a vertex of degree $j$ with probability $j \tilde{N}_j/\sum_k k \tilde{N}_k$. Denoting the last variable by $\tilde{Z}$, and note that
    \begin{align}
       \sum_j j \left| \tilde{N}_j  -  n p_j \right| &\leq \Delta \sum_j \left(\left| \tilde{N}_j - N_j\right| + \left| N_j - np_j\right|\right) \nonumber\\
       &\leq \Delta \sum_{i\leq\ell_1} |S_i| + n^{1/2+o(1)} \lesssim
       \Delta n \exp\left(-\tfrac13 \gamma\sqrt{\log n}\right)\,,\label{eq-tildeNj-npj}
    \end{align}
    where the inequality  between the lines used~\eqref{eq-Nk-1}--\eqref{eq-Nk-2}, and the last inequality used~\eqref{eq-explored-v}. Thus, 
      \begin{align*}
      \|\P(Z\in\cdot)-\P(\tilde{Z}\in\cdot)\|_{\tv} &= \frac{1}2 \sum_{j \leq \Delta}
      \bigg| \frac{ j p_j}{\sum_k k p_k} -\frac{ j \tilde{N}_j}{\sum_k k\tilde{N}_k}\bigg | \\
      &= \frac{1}2 \sum_{j \leq \Delta} j
      \bigg| \frac{  p_j \sum_k k\tilde{N}_k  - \tilde{N}_j  \sum_k k p_k }{ \sum_k k p_k \sum_k k\tilde{N}_k}\bigg | \,,
      \end{align*}
   which, if we write $\Xi = \sum_k k ( \tilde{N}_k - n p_k )$, is equal to
      \begin{align*}
      \frac12 \sum_{j\leq\Delta}  & j\bigg| \frac{ n p_j - \tilde{N}_j}{ n\sum_k k p_k} + \frac{\tilde{N}_j \,\Xi }{n\sum_k kp_k \sum_k k \tilde{N}_k}\bigg| \leq \frac1{2n}\sum_j j | \tilde{N}_j - n p_j | + |\Xi| \sum_{j\leq\Delta} \frac{\tilde{N}_j}{n^2}\,,
      \end{align*}
   for large $n$, using $\sum_k k \tilde{N}_k \geq (1-o(1))n$.
   By~\eqref{eq-tildeNj-npj} and the hypothesis on $\Delta$,
     \begin{align*}
      \|\P(Z\in\cdot)-\P(\tilde{Z}\in\cdot)\|_{\tv}  \lesssim (\Delta+\Delta^2) \exp\left(-\tfrac13 \gamma\sqrt{\log n}\right)\,.
       \end{align*}
     A union bound over $\ell_1\Delta^{R+1}$ trials of examining $\bdeg(v)$ shows that this contribution to $\P(\widehat\cB)$ is also $o(1)$.
\end{enumerate}

\subsection{Lower bound on mixing time in Theorem~\ref{mainthm:deg-seq}}
Set \[ t^-=\frac{\ell_0}{\spe} = \spe^{-1}\left(\ell_1 -\tfrac1{4\dimH}\gamma\sqrt{\log n}\right) \,,\]
and observe that~\eqref{eq-speed-epsilon-bnd} implies that
$\dist(\rho, X_{t^-}) \leq \ell_1-(8\dimH)^{-1}\gamma\sqrt{\log n}$
with probability at least $1-\epsilon$.
By Lemma~\ref{lem-detour}, applied in a union bound over $t\leq t^-=O(\log n)$, w.h.p.\
\begin{equation}
  \label{eq-dist-xt-xi-lower}
  \max\{\dist(X_t,\xi): t \leq t^-\} \leq R
\end{equation}
for a large enough $c_0>0$ (depending only on the distribution of $Z$) in the definition~\eqref{eq-R-def} of $R$. It follows that upon a successful coupling, $\sX_{t^-}$ is within distance $R=O(\log\log n)$ from the set of vertices corresponding to $\cup_{i\leq \ell_1} S_i$, the vertices exposed in $\ell_1$ rounds of our truncated exploration process of the GW-tree. Thus, by~\eqref{eq-explored-v}, the distribution of $\sX_{t^-}$ given a successful coupling has a mass of $1-\epsilon$ on at most
\[\Delta^R \sum_{i\leq \ell_1}|S_i|=o(n)\] vertices, yielding the required lower bound.\qed

\subsection{Upper bound on mixing time in Theorem~\ref{mainthm:deg-seq}}
\subsubsection{The case of minimum degree 3 (Theorem~\ref{mainthm:deg-seq},  Part~\eqref{it-deg-seq-1})}
Set
\[ t^+ = \frac{ \ell_0 +\ell_1}{2\spe}
= \spe^{-1} \left( \ell_1 - \tfrac{1}{8\dimH}\gamma\sqrt{\log n}\right) \,.\]
We claim that, if $P_G$ is the transition matrix of the SRW $(\sX_t)$, then for large $n$,
\begin{equation}
  \label{eq-truncated-in-G}
\P\Bigg(\sum_{v\in V(G)} \bigg[ P_G^{t^+}(v_1,v) \wedge \frac1n \exp\left(\tfrac43\gamma\sqrt{\log n}\right)\bigg]
\geq 1-\sqrt{5\epsilon}\Bigg) \geq 1-\sqrt{5\epsilon}  \,.
\end{equation}
To see this, first define the event $\Upsilon = \Upsilon_1 \cap \Upsilon_2 \cap \Upsilon_3$, where
\begin{align*}
   \Upsilon_1 &= \left\{ \ell_0 + \tfrac1{16\dimH}\gamma\sqrt{\log n}\leq \dist(\rho, X_{t^+}) \leq \ell_1 - \tfrac1{16\dimH}\gamma\sqrt{\log n} \right\}\,,\\
   \Upsilon_2 &= \left\{(\sX_t,X_t)_{t=1}^{t^+}\mbox{ are successfully coupled}\right\}\,,\\
   \Upsilon_3 &= \left\{ \xi_{\ell_0} = \left(\mbox{loop-erased trace of $(X_t)_{t\leq t^+}$}\right)_{\ell_0} \right\} \cap  \Big\{ \xi_{\ell_0} \in S''_{\ell_0}\Big\} \mbox{ for $(X_t)$ on $T$}\,,
\end{align*}
where $S''_{\ell_0}$ is as defined in~\eqref{eq-xi-t-upper}.
Recall from~\eqref{eq-speed-epsilon-bnd} that $\P(\Upsilon_1^c) \leq \epsilon$;
we established in~\S\ref{sec:coupling-to-GW} that
 $\P(\Upsilon_2^c \cap \Upsilon_1) \leq \epsilon+o(1)$ (utilizing the upper bound portion of the event $\Upsilon_1$); and as for $\Upsilon_3$, the first event in that intersection occurs w.h.p.\ on the event of the lower bound on $\dist(\rho,X_{t^+})$ from $\Upsilon_1$ (otherwise there would be some $t \leq t^+$ at which $\dist(X_t,\xi)\gtrsim \sqrt{\log n}$, which is unlikely as argued above~\eqref{eq-dist-xt-xi-lower}), and therefore $\P(\Upsilon_3^c \cap \Upsilon_1) \leq 2\epsilon+o(1)$ by the bound above~\eqref{eq-xi-t-upper}. Overall,
 \[ \P(\Upsilon^c) \leq 4\epsilon+o(1)\,,\]
 and by Markov's inequality, for sufficiently large $n$ we have
 \[ \P\Big(\P_G(\Upsilon) \geq 1-\sqrt{5\epsilon}\Big) \geq 1-\sqrt{5\epsilon}\,.\]
 In particular, with the definition of $S''_{\ell_0}$ in mind, with probability at least $1-\sqrt{5\epsilon}$,
\begin{align*}
  \sum_{v\in V(G)} &\bigg[ \P_G\left(\sX_{t^+}=v \,, \Upsilon\mid \sX_0=v_1\right)   \wedge \frac1n \exp\left(\tfrac43\gamma\sqrt{\log n}\right)\bigg] \\
  &\geq
  \sum_{v\in V(G)}  \P_G\left(\sX_{t^+}=v \,, \Upsilon\mid \sX_0=v_1\right)
  \geq  \P_G(\Upsilon) \geq 1-\sqrt{5\epsilon}\,,
\end{align*}
and~\eqref{eq-truncated-in-G} follows.

We now reveal $G$, and can assume that the event in~\eqref{eq-truncated-in-G} occurs. That is, if we set $\mu(v) = P_G^{t^+}(v_1,v) \wedge \frac1n \exp\left(\tfrac43\gamma\sqrt{\log n}\right)$ then $1-\mu(G) \leq \sqrt{5\epsilon}$.
With the fact $\pi(v) \geq 1/(\Delta n)$ in mind, we infer that
\[\|P_G^{t^+ +s}(v_1,\cdot) - \pi\|_{\tv} \leq \sqrt{5\epsilon} + \|\mu P_G^s - \pi\|_{\tv} \leq
\sqrt{5\epsilon} + \|\mu P_G^s - \pi\|_{L^2(\pi)}\,,
\]
and
\[
 \left\|\mu/\pi-\mu(G)\right\|_{L^2(\pi)} \\
\leq \exp\left(\left(\tfrac43+o(1)\right)\gamma\sqrt{\log n}\right)\,.
\]
In the setting where $p_1=p_2=0$ we can now conclude the proof, using the well-known fact that $G$ is then w.h.p.\ an expander and by Cheeger's inequality the spectral gap of the SRW on $G$, denoted $\texttt{gap}$, is uniformly bounded away from 0 (see~\cite{AlonSpencer08} on expanders and, \emph{e.g.},~\cite[Lemma~3.5]{DKLP11} (treating the kernel of the giant component) for the standard argument showing expansion for a random graph with said degree sequence). Indeed, if we take
\begin{equation}\label{eq-s-def}
 s = c\, \texttt{gap}^{-1} \left(\gamma\sqrt{\log n} + \log(1/\epsilon)\right)
\end{equation}
for a large enough absolute constant $c>0$, then we will have
\[  \|P_G^{t^+ +s}(v_1,\cdot) - \pi\|_{\tv} \leq
 \sqrt{5\epsilon} + \epsilon\,,\]
as desired.\qed

\subsubsection{Allowing degrees 1 and 2 (Theorem~\ref{mainthm:deg-seq}, Part~\eqref{it-deg-seq-2})}\label{subsec-deg1-deg2}
The general case requires one final ingredient, since $G$ is no longer an expander. 
Instead of $G$, consider the graph $G'$ on $n'\leq n$ vertices where every path of degree-2 vertices whose length is larger than $R=c_0\log\log n$ is contracted into a single edge, and similarly, every tree whose volume is at least $R$ is replaced by a single vertex.

Let $U$ be the set of vertices in $V(G)\setminus V(G')$ as well as their neighbors in $G$.
We claim that, if $v$ is a uniform vertex conditioned to belong to $\cC_1(G)$ and
\[ \cB_{U,t} = \bigcup_{k\leq t}\{X_k\in U\}\,,\]
where $X_k$ is SRW from $v$, then $\P_v(\cB_{U,\log^2 n}) = o(1)$.
This will reduce the analysis to $G'$, as clearly on the event $\cB_{U,t}^c$ we can couple the walks on $G$ and $G'$ up to time $t$.
Since $\cC_1$ is linear w.h.p., it suffices to take $X_0$ to be uniform over $V$, {\it i.e.}, to show that
\begin{align}
  \label{eq-But,bound}
  \frac1n\sum_{v\in V} \P_v( \cB_{U,\log^2 n}) = o(1)\,.
\end{align}
The fact that $p_2$ is uniformly bounded away from $1$ implies that
\begin{equation}
  \label{eq-v-in-U-given-degs}
  \P(v\in U \mid (D_u)_{u\in V}) < D_v (\log n)^{-10}
\end{equation}
 for every vertex $v$ provided we choose the constant $c_0$ to be sufficiently large (for the long paths this uses the  decay of $\P(Y > R)$ where $Y\sim \Geom(p_2)$, whereas for the hanging trees this uses tail estimates for the size of subcritical GW-trees (see~\cite{Dwass69})).

Condition on the degree sequence $(D_u)_{ u\in V}$, let $N_k$ denote the number of vertices of degree $k$ in $G$,
and assume~\eqref{eq-Nk-1} holds, neglecting an event of probability $o(1)$ as described above that inequality.
Since $\sum_k k p_k = O(1)$ thanks to our assumption on $\E Z$, it follows (see~\eqref{eq-Nk-1}--\eqref{eq-Nk-2}) that
$\sum_{u} D_u = (1+o(1))n \sum_k k p_k = O(n)$. In particular, for every $v\in V$ we have $\pi(v)^{-1} = (\sum_u D_u)/D_v = O(n)$. Therefore, for every $t$,
 \begin{align*}
   \frac1n& \sum_v \P_{v}\left(\cB_{U,t}\mid (D_u)_{u\in V}\right) \lesssim \sum_v \pi (v) \P_{v}\left(\cB_{U,t}\mid (D_u)_{u\in V}\right) \\ &= \E\bigg[ \sum_v \pi(v) \P_v\Big( \bigcup_{k\leq t} \{X_k\in U\} \mid G\Big)\;\Big|\; (D_u)_{u\in V}\bigg]
   \leq t \sum_v \pi(v) \P\left(v\in U \mid (D_u)_{u\in V}\right)\,,
 \end{align*}
 where the last step combined a union bound with the stationarity of the SRW. Let
$ A = \{v : D_v \geq \log^3 n \}$,
and notice that $\sum_{k\geq \log^3 n} k p_k = O(1/\log^{3} n)$ using $\E Z =  O(1)$, which, by~\eqref{eq-Nk-1}--\eqref{eq-Nk-2}, yields $\pi(A) =O(1/ \log^{3} n)$. The above display then gives
 \begin{align*}
   \frac1n \sum_v \P_{v}\left(\cB_{U,t}\mid (D_u)_{u\in V}\right) &\lesssim  t \pi(A) + t\frac{\log^3 n}n \sum_{v\in A^c} \P(v\in U \mid (D_u)_{u\in V})  \lesssim \frac{t}{\log^4 n}\,,
 \end{align*}
where the last inequality used~\eqref{eq-v-in-U-given-degs} to replace each summand in the sum over $v\in A^c$ by $D_v (\log n)^{-10} \leq 1/\log^7 n$. This establishes~\eqref{eq-But,bound}.

The proof is concluded by noting that the kernel of $G'$ is an expander w.h.p., and the effect of expanding some of its edges into paths of length at most $R$, or hanging trees of volume at most $R$ on some of its vertices, can only decrease its Cheeger constant to $c/R$ for some fixed $c>0$. Hence, by Cheeger's inequality, the spectral gap of $G'$ satisfies
$ \texttt{gap} > c/R^2 $
for a fixed $c>0$ (depending only on the degrees of $G'$), and thereafter taking $s$ as in~\eqref{eq-s-def} (noting that $s\asymp\sqrt{\log n}(\log\log n)^2$ now) establishes the following: for every $\epsilon>0$ there exists some $\gamma>0$ such that, with probability at least $1-\epsilon-o(1)$,
\[
 d^{(v_1)}_{\tv}\big(t_\star- \gamma w'_n\big)> 1-\epsilon\,,\qquad \mbox{whereas}\qquad d^{(v_1)}_{\tv}\big(t_\star+ \gamma w'_n\big)<\epsilon\,,\]
where $w_n = \sqrt{\log n}(\log\log n)^2$. This immediately implies cutoff, w.h.p., for any choice of a sequence $w_n$ such that $w'_n =o( w_n)$.
\qed

\subsection{Random walk on the \texorpdfstring{Erd\H{o}s-R\'enyi}{Erdos-Renyi} random graph: Proof of Theorem~\ref{mainthm-giant}}
The proof will follow from a modification of the argument used for a random graph on a prescribed degree sequence, with the help of the following structure theorem for $\cC_1$.
\begin{theorem}[{\cite[Theorem~1]{DLP14}}]\label{thm-struct}
Let $\cC_1$ be the largest component of $\cG(n,p)$ for $p = \lambda/n$ where $\lambda>1$ is fixed.
 Let $\mu<1$ be the conjugate of $\lambda$, {\it i.e.},
$\mu e^{-\mu} = \lambda e^{-\lambda}$. Then $\cC_1$ is contiguous to the following model $\tilde{\cC}_1$:
\begin{enumerate}[\indent 1.]
  \item\label{item-struct-base} [Kernel] Let $\Lambda \sim \cN\left(\lambda - \mu, 1/n\right)$, take $D_u \sim \Poi(\Lambda)$ for $u \in [n]$ i.i.d. and condition that $\sum D_u \one_{D_u\geq 3}$ is even.
  Let $N_k = \#\{u : D_u = k\}$ and $N= \sum_{k\geq 3}N_k$.
  Draw a uniform multigraph $\cK$ on $N$ vertices with $N_k$ vertices of degree $k$ for $k\geq 3$.
  \item\label{item-struct-edges} [2-Core] Replace the edges of $\cK$ by paths of i.i.d.\ $\Geom(1-\mu)$ lengths. 
  \item\label{item-struct-bushes} [Giant] Attach an independent $\Poi(\mu)$-Galton-Watson tree to each vertex.
\end{enumerate}
That is, $\P(\tilde{\cC}_1 \in \cA) \to 0$ implies $\P(\cC_1 \in \cA) \to 0$
for any set of graphs $\cA$.
\end{theorem}

The first point is that we develop the neighborhood of a fixed vertex of $G$ (not of $\cC_1$) as a GW-tree with offspring distribution $Z\sim \Poi(\lambda)$ for $\lambda>1$ fixed. Should its component turn out to be subcritical, we abort the analysis (with nothing to prove).

The second point is that, when analyzing the probability of failing to couple the SRW on $G$ vs.\ the GW-tree, mismatched degrees are no longer related to sampling with/without replacement from a degree distribution, but rather to the total-variation distance between $\Poi(\lambda)$ and $\Bin(n',\lambda/n)$ where $n'$ is the number of remaining (unvisited) vertices. By~\eqref{eq-explored-v} we know that $n' \geq n (1-\exp(-\frac13\gamma\sqrt{\log n}))$, and therefore, by the Stein-Chen method (see, \emph{e.g.},~\cite[Eq.~(1.23)]{BHJ}), this total-variation distance is at most
$\lambda/n  + \left\|\Poi(\lambda)- \Poi(\lambda n'/n)\right\|_{\tv} \leq (\log n)^{-10}$ (with room to spare).

The final and most important point involves the modification of $G$ into $G'$, where the spectral gap is at least of order $R^{-2}$ where $R = c_0 \log\log n $ for some suitably large constant $c_0>0$.
We are entitled to do so, with the exact same estimates on the probability that SRW visits the set $U$ of contracted vertices, by Theorem~\ref{thm-struct} (noting the traps introduced in Step~\ref{item-struct-edges} are i.i.d.\ with an exponential tail, and similarly for the traps from Step~\ref{item-struct-bushes}); the crucial fact that $G'$ has a Cheeger constant of at least $c/R$ follows immediately from that theorem using the expansion properties of the kernel.
\qed

\subsection{The typical distance from the origin}
A consequence of our arguments above is the following result on the typical distance of random walk from its origin at $t\asymp \log n$.
\begin{corollary}
  \label{cor-typical-distance}
  Consider random walk $(\sX_t)$ from a uniform vertex $v_1$ in $\cC_1$ either in the setting of Theorem~\ref{mainthm-giant} or of Theorem~\ref{mainthm:deg-seq}, let $\spe$ denote the speed of random walk on the corresponding Galton-Watson tree, and let $\lambda$ be the mean of its offspring distribution. For every fixed $a>0$, if $t=a\log_\lambda n$ then $\dist(v_1,\sX_t)/\log_\lambda n \to \left( \nu a \;\wedge\; 1\right)$ in probability.
\end{corollary}
\begin{proof}
First consider the case $a < \nu^{-1}$, and fix $\epsilon>0$ small enough such that $a \nu < 1-\epsilon$.
 Let $T$ be a GW-tree with an offspring distribution that has the law of $Z$ (this time without any truncation), and consider the coupling described in~\S\ref{sec:coupling-to-GW} of SRW $(X_t)$ on $T$ to SRW $(\sX_t)$ on the subtree $\Gamma\subset T$ explored by the standard breadth-first-search of $G$ from $v_1$. That is, we first reveal $T$ and the walk $(X_1,\ldots,X_t)$, then explore the neighborhood of $v_1$ in $G$, and estimate the probability of $\cup_{k\leq t}\bcyc(X_k)$. By the definition of $\nu$ in~\eqref{eq-nu-def}, the assumption on $a$ implies that $\max_{k\leq t}\dist(\rho,X_k) < (1-\epsilon/2)\log_\lambda n$ w.h.p., and therefore, at all times $k\leq t$, the boundary of the exploration process of $G$ (which is at most the size of $T$) contains at most $n^{1-\epsilon/3}$ half-edges. Recalling the argument from~\S\ref{sec:coupling-to-GW} (part (b)), the probability of hitting a cycle at a given step $X_k$ is therefore $O(\Delta n^{-\epsilon/3}) \leq n^{-\epsilon/4}$ for large enough $n$, and a union bound over $k$ implies that w.h.p.\ $\cup_{k\leq t}\bcyc(X_k)$ does not occur and the coupling is successful. In this case, $\dist(v_1,\sX_t) = \dist(\rho,X_t)$ and the desired result follows from the definition~\eqref{eq-nu-def} of $\nu$.

Next, consider $a \geq \nu^{-1}$, and suppose first that $p_1=p_2=0$. As stated in the introduction, it is then the case that the diameter of the graph is $(1+o(1))\bar{D}$ where $\bar{D} = \log_\lambda n$, and it remains to provide a matching lower bound on $\dist(v_1,\sX_t)$. Let
\[ R = \big\lfloor(\log n)^{1/4}\big\rfloor \]
and
\[ \ell_1 = \left\lfloor (1-\epsilon)\log_\lambda n\right\rfloor\,,\quad \ell_0 = \ell_1 - R\,.\]
Consider the exploration tree $\Gamma$, as defined above, up to level $\ell_1$.
\begin{claim}\label{clm-aux-H}
  Let $H$ be the auxiliary graph on level $\ell_0$ of $\Gamma$, where $(u,v)\in E(H)$ if there is an edge of $G$ connecting $\Gamma_R(u)$ and $\Gamma_R(v)$.
  With high probability, the largest connected component of $H$ has size at most $10/\epsilon$.
\end{claim}
\begin{proof}
As stated above, since the boundary of $\Gamma$ contains at most $n^{1-\epsilon/3}$ half-edges, the analysis of~\S\ref{sec:coupling-to-GW} shows that the probability of $\bcyc(z)$ is at most $n^{-\epsilon/3+o(1)}$ for every vertex $z$ in $\Gamma$ at distance at most $\ell_1$ from $v_1$. Therefore,  the degree of a vertex $u$ in $H$ is dominated by a binomial random variable $\zeta \sim \Bin(n^{\epsilon/8}, n^{-\epsilon/4})$, where we used that $|\Gamma_R(u)| \leq \Delta^R = n^{o(1)}$.
 Consequently, the size of the component of $u$ in $H$ is dominated by the total progeny $\Lambda$ of a branching process with offspring random variables $\zeta$. By exploring this branching process via breadth-first-search (or alternatively, through the Dwass--Otter formula; see, e.g.,~\cite{Dwass69}), one finds that
\begin{align*}
  \P(\Lambda > \ell) &\leq \P\bigg(\sum_{i=1}^\ell (\zeta_i - 1) \geq 0\bigg) \leq \P\left(\Bin(\ell n^{\epsilon/8}, n^{-\epsilon/4}) \geq \ell\right) \leq \sum_{k\geq \ell}\binom{\ell n^{\epsilon/8}}{k}n^{-k\epsilon/4} \\
  &\leq\sum_{k\geq \ell}\bigg(\frac{e \ell n^{\epsilon/8}}k\bigg)^{k} n^{-k\epsilon/4} \leq \sum_{k\geq \ell} \left( e n^{-\epsilon/8}\right)^k = O(n^{-\ell \epsilon/8})\,.
\end{align*}
Taking $\ell > 10/\epsilon$ supports a union bound over $u$ and completes the proof.
\end{proof}
\begin{corollary}\label{cor-aux-H}
  Let $H$ be as defined in Claim~\ref{clm-aux-H}. With high probability, for every connected component $C$ of $H$ there exists an interval $I\subset [\ell_0,\ell_1]$ of length at least $\epsilon R/10$ so that the induced subgraph of $G$ on levels in $I$ of
  $\bigcup_{u\in C}\Gamma_R(u)$ is a forest.
\end{corollary}
Suppose that $G$ satisfies the property of the above corollary. We claim that if $\dist(v_1,\sX_k)=\ell_1$ for some $k$, then w.h.p.\ (w.r.t.\ the walk)
\[ \min_{1\leq j \leq \log^2 n} \dist(v_1,\sX_{k+j}) \geq \ell_0\,.\]
Indeed, in light of the above corollary, in order for the walk to reach level $\ell_0$ starting from level $\ell_1$, it must cross an interval of length $\epsilon R/10$ where every vertex has at least two children in a lower level. The probability that the biased random walk $(S_t)$, where $S_t-S_{t-1}$ is $1$ with probability $\frac13$ and $-1$ with probability $\frac23$, reaches $\epsilon R/10$ before returning to the origin is at most $2^{-\epsilon R/10} < (\log n)^{-5}$, and taking a union bound over the $\log^2 n$ possible time steps now completes the treatment of the case $p_1=p_2=0$.

Finally, when we allow degrees 1 and 2, recall from~\S\ref{subsec-deg1-deg2} that one can modify the random graph $G$ to a graph $G'$ where the long paths and large hanging trees all have size $O(\log\log n)$, and couple the SRW on $G$ to the SRW on $G'$ w.h.p.\ for all $t \leq \log^2 n$. The only change in the above argument is in the analysis of the biased random walk $(S_t)$ along the interval of length $\epsilon R/10$ that corresponds to a forest in $G$. There, the problem reduces to an interval shorter by a factor of $c \log \log n$, and the fact that $2^{-\epsilon R/ (c \log \log n)}$ is again at most $(\log n)^{-5}$ for large enough $n$ concludes the proof.
\end{proof}

\section{Non-backtracking random walk on random graphs}\label{sec:NBRW}

In this section we consider the NBRW on a random graph $G$
   with i.i.d.\ degrees with distribution $(p_i)_{i\geq 0}$ (see the paragraph preceding Theorem~\ref{mainthm:deg-seq}) where $p_0=p_1=0$. Recall that the NBRW is a Markov chain on the set $\vec{E}$ of directed edges, which moves from $e = (u,v)$ to a uniformly chosen edge among $\{ e' = (v,w) : w\neq u\}$.
   Since $\deg(v)\geq 2$ (as $p_0=p_1=0$) this Markov chain is well-defined for all $t$, and that the uniform measure $\pi$  on $\vec{E}$ is stationary. Let $P_G^t(e,e')$ be the $t$-step transition probability of the NBRW.

Our assumption on the degree distribution is that
\begin{equation}\label{eq-momentcondition}
M_0 := \sum_k k ( \log k)^2 p_k  < \infty\,.
\end{equation}
Let $\lambda=\sum_k kp_k $ be the mean degree (noting  that $2 \leq \lambda <\infty$ by~\eqref{eq-momentcondition}), and let
\[
q_{k-1} := k p_k / \lambda\qquad (k \geq 1)\,.
\]
be the shifted size-biased distribution.
We further assume that the random variable $Z$ given by $\P(Z = k)=q_k$  satisfies
\begin{equation}\label{eq-Z-max-degree-hypo-2}
 \P(Z > \Delta_n) = o(1/n) \quad\mbox{ for }\quad \Delta_n := \exp\big[(\log n)^{1/2-\delta}\big]\,.
 \end{equation}
Recall that the dimension of harmonic measure for NBRW in the GW-tree with offspring distribution $(q_k)$ is
\begin{equation}\label{dimension}
\tilde\dimH = \sum_k q_k \log k\,,
\end{equation}
which is finite under~\eqref{eq-momentcondition}.
To state our next result, we let
\[ \tmix^{(\pi)}(\epsilon)=\min\{ t : d^{(\pi)}_{\tv}(t)<\epsilon\} \quad\mbox{ where }\quad d^{(\pi)}_{\tv}(t)=\frac1{|\vec{E}|}\sum_{e\in \vec{E}}\|P_G^t(e,\cdot)-\pi\|_\tv\,.\]
\begin{theorem}[typical starting point]\label{T:mix_average}
Let $G$ be a random graph on $n$ vertices with i.i.d.\ degrees with distribution $(p_k)_{k\geq 0}$, where $p_0=p_1=0$ and~\eqref{eq-momentcondition}--\eqref{eq-Z-max-degree-hypo-2} hold.
For every fixed $0 < \eps< 1$, with probability at least $1-\epsilon$,
\[
\tmix^{(\pi)}(\eps) = \tilde\dimH^{-1} \log n + O( \sqrt{\log n})\,.
\]
\end{theorem}

\begin{theorem}[worst starting point]\label{T:mix_worst}
In the setting of Theorem~\ref{T:mix_average} with the extra assumption $p_2=0$,
for every fixed $0 < \eps< 1$, with probability at least $1-\epsilon$,
\[\tmix(\eps) = {\tilde\dimH}^{-1} \log n + O( \sqrt{\log n})
\,.\]
\end{theorem}

\begin{remark}
After completing this work we learned that another proof of Theorem~\ref{T:mix_worst} was obtained independently by Ben-Hamou and Salez~\cite{Announce} with a weaker assumption on $\Delta$ and an asymptotic estimate of the total-variation distance within the cutoff window.
\end{remark}

\subsection{Mixing from a typical starting point}

\subsubsection{The truncated exploration process}
\label{SS:good}
We first condition on the degree sequence of the random graph, $\{D_v : v \in V\}$, and may assume, as in~\S\ref{sec:coupling-to-GW}, that~\eqref{eq-Nk-1} holds by neglecting an event of probability $o(1)$.
Let $(Y_t)$ denote the NBRW, and for every directed edge $e=(u,v)$, let $D_e = D_v-1$ (the number of directed edges that the NBRW can transition to from $e$).

Throughout this proof, let $t_\star$ be an odd integer given by
\begin{equation}\label{eq-t-L-def}
t_\star =  2 L + 1\quad\mbox{ where }\quad L = \Big\lceil (2{\tilde\dimH})^{-1} \log n + 2K  \tilde\dimH^{-3/2} \sqrt{\log n}\Big\rceil\,,
\end{equation}
for 
\begin{equation}\label{eq-K-def}
K := 8 \sqrt{M_0} / \epsilon\quad\mbox{ with $M_0$ from~\eqref{eq-momentcondition}}\,.
\end{equation}
Let $e=(u,v)\in\vec E$. Let $T_e$ be the tree of depth $L$ constructed as follows. Explore the neighborhood of $e$ via breadth-first-search (as explained in~\S\ref{sec:coupling-to-GW}). Exclude every half-edge that is matched to a previously encountered half-edge from $T_e$. Moreover, do not explore an edge $e'$ (truncate its subtree) if it satisfies
\begin{equation}\label{subtree}
\Big|\sum_{e_0 \in \Psi(e,e')} (\log D_{e_0}  - \tilde\dimH)\Big| \ge K\sqrt{L}\,,
\end{equation}
where $\Psi(e,e')$ denotes the directed path from $e$ to $e'$ in $T_e$.
Let $\partial T_e$ denote the half-edges at level $L$ that had been encountered (and not yet matched) in this process.

Next, consider a second edge $f = (x,y)\in \vec E$ such that $x,y \notin T_e$. Define $T_{f}^{e}$ to be the tree of depth $L$ formed by exploring  $T_{f^*}$  corresponding to the \emph{reversed} edge $f^* = (y,x)$, as described above, with the following modification: if a half-edge is matched to  $\partial T_e$, we do not include it in the tree.
Set
\[\cF_{e} = \sigma( T_e) \,,\quad \cF_{f} = \sigma( T_f^e)\,,\quad \cF_{e,f} = \sigma(\cF_e , \cF_f)\,.\]
Consider the sequence $(a_k)$ defined by
\begin{equation}\label{eq-ak-def}
a_k := (2  M_0)^{1/2} / k \quad\mbox{ with $M_0$ from~\eqref{eq-momentcondition}}\,,
\end{equation} and  let $\cE$ denote the set of all edges $e\in\vec E$ such that the NBRW from $e$ remains in $T_e$ except with probability $a_K$, for $K$ the constant in~\eqref{eq-K-def}; that is, define
 \[ \cE  =\left\{ e \in \vec{E} \,:\; Q_{e} \leq  a_K\right\} \quad\mbox{where}\quad Q_{e}:=\P_G \Big( \bigcup_{t \leq L} \big\{Y_t \notin \vec E(T_e) \big\}\given Y_0 = e\Big)\,.
 \]
Analogously, let
 \[ \cE_e  =\left\{ f \in \vec{E} \,:\; Q_{e,f}\leq  a_K\right\} \quad\mbox{where}\quad Q_{e,f}:=\P_{G} \Big( \bigcup_{t \leq L} \big\{Y_t \notin \vec E(T_f^e) \big\}\given Y_0 = f^*\Big)\,.
 \]
Note that, by labeling the half-edges of each vertex, we may refer to a directed edge $e=(u,v)$ also as $e=(i,v)$ where $i\in\{1,\ldots, D_{v}\}$ is the label of a  half-edge matched to a half-edge of the vertex $u$.
Similarly, we may refer to $f=(x,y)$ as $f=(x,j)$ where $j$ corresponds to a half-edge $j\in\{1,\ldots, D_v\}$ matched to a half-edge of $y$.

With this notation, each half-edge in $\partial T_e$ of the form $(u',i)$ for some $u'\in V$ and $i\in\{1,\ldots,D_{u'}\}$ corresponds to a unique edge in $\vec E$, which we include in $\vec E(T_e)$, the set of directed edges induced on $T_e$.
We claim that, for each $1\leq k \leq L$, 
\begin{equation}\label{eq-Te-size-level-k}\#\left\{ e'\in \vec E(T_e) \,:\; \dist(e,e')=k\right\} \leq n^{1/2+o(1)}\,.\end{equation}
 Indeed, by condition~\eqref{subtree}, and using its notation as well as~\eqref{eq-t-L-def}, for every such $e'$,
\begin{align}\label{eq-Yk=e-bound} \P_e(Y_k = e') &= \prod_{e_0\in\Psi(e,e')} \frac1{D_{e_0}}  \geq e^{ -k\tilde\dimH - K\sqrt{L}}
\geq e^{- L \tilde\dimH - K\sqrt{L}} 
 = n^{-1/2-o(1)}\,.
\end{align}
This implies~\eqref{eq-Te-size-level-k} since the events $\{Y_k=e'\}$ are disjoint for all edges $e'$. In particular, using that $\Delta_n=n^{o(1)}$ and $L=O(\log n)$, the total number of half-edges encountered in the exploration process of $T_e$ is at most $n^{1/2+o(1)}$, and, similarly, this also holds for $T_f^e$.

We next estimate the probability that $e\in\cE$ and $f\in\cE_e$ for some fixed $e,f$.
\begin{lemma}\label{L:exit}
For every two edges  $e=(i,v)$, $f=(x,j)$ (with $1\leq i\leq D_{v}$, $1 \leq j \leq D_x$),
\[
\P(e \notin  \cE ) \leq a_K\,,\qquad
\P( f\notin \cE_e \mid T_e) \leq a_K\,.
\]
\end{lemma}

\begin{proof}
The  argument in~\S\ref{sec:coupling-to-GW} shows that the NBRW $(Y_t)_{t\geq 0}$ from $Y_0=e$ on the random graph $G$ can be coupled  to a NBRW $(\xi_t)_{t \geq 0}$ on a (non-truncated) GW-tree $T$ w.h.p., as only the events $\bcyc$ and $\bdeg$ are to be avoided  (\emph{avoiding a cycle} and \emph{coupling vertex degrees}, respectively)  when no truncation is performed on the tree. 

On the GW-tree, the random variables $ (D_{\xi_i})_{i\geq 1}$, recording the number of offspring of each site visited by the NBRW,  are i.i.d.\ with distribution given by $(q_k)_{k\geq 0}$, and in particular, $\E[ \log D_{\xi_i}] = \tilde \dimH$ and $\E [\log^2(D_{\xi_i})] \leq M_0$ for $M_0$ from~\eqref{eq-momentcondition}. Hence,
\[
\P\bigg(\max_{1\leq k \leq L} \Big|\sum_{i=1}^k(\log D_{\xi_i}  - \tilde\dimH ) \Big| \geq K \sqrt{L} \bigg) \le \frac{M_0}{K^2}
\]
by Kolmogorov's  Maximal Inequality. Hence, the above coupling of $(\xi_i)$ and $(Y_i)$, which succeeds with probability $1-o(1)$, implies that
\[\E [Q_{e}] \leq M_0/K^2 + o(1) = (\tfrac12+o(1)) a_K^2 \,.\]
It thus follows from Markov's inequality that
\begin{align*}
\P\left( Q_{e}\geq a_K\right) \leq \frac{\E\left[ Q_{e}\right]}{a_K}   \leq (\tfrac12+o(1)) a_K  < a_K\,,
\end{align*}
provided that $n$ is large enough. The statement concerning the NBRW from $f^*$ follows in the exact same manner.
\end{proof}

\subsubsection{Estimating the distribution of the walk via two truncated trees}
We wish to bound the quantity $P_G^t(e,f)$ from below.
Note that, since $t_\star=2L+1$,
\[
P_G^{t_\star}(e,f) \geq  \sum_{e_0\in \partial T_e}\sum_{ f_0 \in \partial T_f^e} P_G^L(e, e_0) \one_{\{e_0=f_0\}} P_G^{L}(f_0,f) \,,
\]
where the event $\{e_0=f_0\}$ for $e_0=(u_0,i)\in\partial T_e$ and $f_0=(x_0,j)\in \partial T_f^e$ is equivalent to having the $i$-th half-edge incident to $x$ be matched with the $j$-th edge incident to $y$.
Further note that $P_G^L\big((j,x_0),(v,v')\big)= P_G^L\big((v',v),(x_0,j)\big) $. Hence,
\[
P_G^{t_\star}(e,f) \geq  \sum_{e_0\in \partial T_e}\sum_{f_0 \in \partial T_f^e} \bar P_{T_e}^{L}(e,e_0) \one_{\{e_0=f_0\}} \bar P^{L}_{T_f^e}(f^*,f_0^*) =: Z_{e,f}\,,
\]
where $\bar P_{T}(\cdot, \cdot)$ corresponds to the walk restricted (not conditioned) to remain in $\vec E(T)$.
We will therefore estimate $\E[ Z_{e,f}]$ and then its variance. 

\begin{lemma} In the setting of Lemma~\ref{L:exit}, let  $A =  \{e\in\cE \}\cap\{f\in\cE_f\}$, which is $\cF_{e,f}$-measurable. Then for every sufficiently large $n$,
\label{L:expZ-varZ}
\begin{align}\label{eq:expZ} 
\E\left[ Z_{e,f}\mid \cF_{e,f}\right] &\geq \frac{1-2a_K}{|\vec E|} \one_A\,,\\
\label{eq:varZ}
\var\left( Z_{e,f}\mid \cF_{e,f}\right) &\leq  n^{-2} \exp\left( -K \tilde\dimH^{-1/2}\sqrt{ \log n} \right)\,. 
\end{align}
\end{lemma}
\begin{proof}
Recall that $\P(e_0=f_0\mid \cF_{e,f})=1/(|\vec E|-m_{e,f})$, where the $\cF_{e,f}$-measurable random variable $m_{e,f}$ denotes the number of half-edges matched during the exploration process in $\cF_{e,f}$, which satisfies $m_{e,f} \leq n^{1/2+o(1)}$ by~\eqref{eq-Te-size-level-k} and the discussion below it.
Since $\bar P_{T_e}^{L}(e,e_0)$ and $\bar P_{T_f^e}^{L}(f,f_0)$ are both $\cF_{e,f}$-measurable, we find that
\begin{align*}
\E\left[ Z_{e,f} \mid \cF_{e,f}\right] &=  \frac{1}{|\vec E|-m_{e,f}} \bigg(\sum_{e_0 \in \partial T_e} \bar P_{T_e}^{L}(e,e_0) \bigg) \bigg(\sum_{f_0 \in T_f^e} \bar P_{T_f^e}^{L}(f^*,f_0^*) \bigg) \\
&\geq \frac1{|\vec E|}  (1-Q_e)(1- Q_{e,f}) \geq \frac1{|\vec E|} (1-2a_k)\one_A\,,
\end{align*}
using the definitions of $e\in \cE$ and $f\in\cE_e$ and that $(1-a_K)^2 \geq 1- 2a_K$. This yields~\eqref{eq:expZ}.

To establish~\eqref{eq:varZ}, note that by definition, 
\begin{align*}
\var \left( Z_{e,f} \mid \cF \right) \leq
 \sum_{e_0,e_1 \in \partial T_e} \sum_{f_0,f_1 \in \partial T_f^e} &\bar P_{T_e}^{L}(e,e_0)  \bar P_{T_e}^{L}(e,e_1) 
 \bar P_{T_f^e}^{L}(f^*,f_0^*) \bar P_{T_f^e}^{L}(f^*,f_1^*)\\
 \cdot& \cov( \one_{ \{e_0=f_0\}}, \one_{ \{e_1=f_1\}} \mid\cF)\,.
\end{align*}
Clearly, if $e_0\neq e_1$ or $f_0\neq f_1$ then the term $\cov( \one_{ \{e_0=f_0\}}, \one_{ \{e_1=f_1\}} \mid\cF)$ equals
\begin{align*}
\P( e_0= f_0 \mid \cF) \Big( \P( e_1 = f_1 &\mid \cF\,,\,e_0=f_0) - \P( e_1=f_1 \mid \cF) \Big)\\
 & \leq \frac{1}{|\vec E| - m_{e,f}} \bigg( \frac{1}{|\vec E| - m_{e,f}-2 } - \frac{1}{|\vec E| -m_{e,f}}\bigg) = \frac{2+o(1)}{|\vec E|^3} \,,
\end{align*}
where in the last inequality we used that $m_{e,f} = o(n) $ while $|\vec E|\geq 2n$.
On the other hand, for the diagonal terms where $e_0=e_1$ and $f_0=f_1$, we have
\[
\var \left(\one_{\{e_0=f_0\}} \mid \cF_{e,f}\right) \leq \P\left( e_0=f_0 \mid \cF_{e,f}\right) = \frac{1+o(1)}{|\vec E|} \,.
\]
Overall, using $|\vec E|\geq 2n$, we conclude that
\begin{align*}
\var \left(Z_{e,f} \mid \cF_{e,f}\right)& \leq \frac{1+o(1)}{4n^3} +
\frac{1+ o(1)}{2n} \bigg( \max_{e_0 \in \partial T_e}\bar P_{T_e}^{L}(e,e_0) \bigg)\bigg( \max_{f_0 \in \partial T_f^e}  \bar P_{T_f^e}^{L}(f^*,f_0^*)  \bigg)\,.
\end{align*}
Now, if $e_0 \in \partial T_e$, then the criterion~\eqref{subtree} implies that
\begin{align*}
\label{max_harm} \bar P^{L}_{T_e}(e,e_0) & \leq e^{-L \tilde\dimH + K\sqrt{L}} = \frac1{\sqrt{n}} \exp\left[-\left(2 - \tfrac1{\sqrt{2}}+o(1)\right)K \tilde\dimH^{-1/2} \sqrt{\log n}\right]  \\
&\leq \frac1{\sqrt{n}} e^{- K \tilde \dimH^{-1/2}\sqrt{\log n}}
\end{align*}
for $n$ sufficiently large, and the same holds for $\bar P^L_{T_f^e}(f^*,f_0^*)$. 
Consequently,
\[
\var \left( Z_{e,f} \mid\cF_{e,f}\right) \leq    \frac{1 + o(1)}{2 n^2} \exp\left( -K \tilde\dimH^{-1/2}\sqrt{ \log n}  \right)  \,,
\]
and the result follows.
\end{proof}

\begin{remark}
The expectation estimate~\eqref{eq:expZ} remains valid  for all smaller $L$; it is for the variance estimate~\eqref{eq:varZ} to hold that we must choose $L$ as in~\eqref{eq-t-L-def}.
\end{remark}

\begin{corollary}\label{cor:Z}
In the setting of Lemma~\ref{L:exit}, for every sufficiently large $n$,
\[
\P\left(Z_{e,f} < (1-3a_K)/|\vec E|\,,\, f\in\cE_e \given T_e\right) \one_{\{e\in \cE\}} \lesssim \exp\left(-K \tilde\dimH^{-1/2}\sqrt{\log n}\right)\,.
\]
\end{corollary}
\begin{proof}
Since $|\vec E|=O(n)$, we deduce from Lemma~\ref{L:expZ-varZ} via Chebyshev's inequality
that
\begin{align}
\P\bigg( Z_{e,f} < \frac{1-3a_K}{|\vec E|} \,\Big|\,\cF_{e,f} \bigg)\one_A &\leq \frac{e^{- K \tilde\dimH^{-1/2}\sqrt{\log n}} }{n^2 (a_K/|\vec E|)^2} \lesssim e^{- K \tilde\dimH^{-1/2}\sqrt{\log n}}
\,.\label{concZ}
\end{align}
Taking expectation given $T_e$ yields the desired result.
\end{proof}

\subsubsection{Upper bound on the mixing time in Theorem~\ref{T:mix_average}}
Let $N=|\vec E|$ and  $e=(i,v)\in \vec E$. Recalling that $P_G^t(e,f) \geq Z_{e,f}$, it follows that
\begin{align*}
&\E\left[ d^{(e)}_\tv(t_\star) \mid T_e \right] \one_{\{e\in\cE\}} = \E\bigg[\sum_f \Big| \frac{1}N - P_G^{t_\star}(e,f)\Big|^+ \bigg] \one_{\{e\in\cE\}} \\
&\leq  \frac{\one_{\{e\in\cE\}}}N \E \big[ \#\{f: f\notin \cE_e\}\given T_e\big] + \frac{\one_{\{e\in\cE\}}}N \E\bigg[\#\Big\{  f \in\cE_e: Z_{e,f}<\frac{1-3a_K}N\Big\}\given T_e \bigg]  + 3a_K\,.
\end{align*} 
The first term in the right-hand  is at most  $a_K$ by Lemma~\ref{L:exit}, and the second is  $o(1)$ by  Corollary~\ref{cor:Z}. 
Therefore, 
\begin{equation}\label{eq-dtv-E}
\E\left[ d^{(e)}_\tv(t_\star) \mid T_e \right]\one_{\{e\in\cE\}}  \leq 4a_K + o(1)\,.
\end{equation} 
Taking expectation over a uniform $e\in\vec E$ and using Lemma~\ref{L:exit} once again, we get
\[ \E\left[ d^{(\pi)}_\tv(t_\star) \right]  \leq 5a_K + o(1) < \epsilon \]
by the choices of $a_K$ and $K$.
\qed

\subsubsection{Lower bound on the mixing time in Theorem~\ref{T:mix_average}}
Now consider
\[
L' = \Big\lfloor  {\tilde\dimH}^{-1} \log n - 2K'  \tilde\dimH^{-3/2} \sqrt{\log n}\Big\rfloor\,,\quad\mbox{ where }\quad
K' = \sqrt{ M_0 / \epsilon}\,.
\]
Define the truncated tree $T_e$ as in~\S\ref{SS:good} (with $K',L'$ replacing $K,L$).
As in~\eqref{eq-Yk=e-bound}, 
\[ \P_e(Y_k = e')  \geq e^{ -k\tilde\dimH - K'\sqrt{L'}}
\geq \frac1n e^{(K'-o(1))\tilde\dimH^{-1/2}\sqrt{\log n}}
\]
for every $1 \leq k \leq L'$, and therefore
\[ 
\#\left\{ e'\in \vec E(T_e) \,:\; \dist(e,e')=k\right\} \leq n e^{-(K'-o(1))\tilde\dimH^{-1/2}\sqrt{\log n}}\quad\mbox{ for each $1\leq k \leq L'$}\,.
\]
Since the coupling argument of~\S\ref{sec:coupling-to-GW}  was valid as long as the size of the truncated tree did not exceed $n \exp(-c\sqrt{\log n})$ for some fixed $c>0$, we can repeat the analysis in the proof of Lemma~\ref{L:exit} to find that
\[
\P \Big( \bigcup_{t \leq L'} \big\{Y_t \notin \vec E(T_e) \big\}\given Y_0 = e\Big) \leq \frac{M_0}{K'^2}+o(1) = \epsilon+o(1)\,,
\]
where the probability is over $G, (Y_t)$. The fact  $|\vec E(T_e)| = o(n)$ completes the proof.\qed

\subsection{Mixing from a worst starting point with minimum degree 3}
We now describe how to modify the proof of Theorem~\ref{T:mix_average} and derive the proof of Theorem~\ref{T:mix_worst}.  Since the lower bound applies also to the worst starting point, it remains to adapt the proof of the upper bound.

Recalling~\eqref{eq-dtv-E}, we have in fact established not only an upper bound on $\E[ d^{(\pi)}_\tv(t) ]$  a uniform starting edge $e$, but rather from every edge $e\in \cE$, conditional on $T_e$. Denote by $T^k_e$ the depth-$k$ truncated tree from~\S\ref{SS:good}, take $t_\star$ and  $L$ as defined in~\eqref{eq-t-L-def} (so that $T^L_e$ is $T_e$ from that section), take $K$ as in~\eqref{eq-K-def}
 and $R = 2\log_2 \log n$. Further denote by $T_e^{R;L}$ the tree obtained by first exposing $T_e^R$, then exposing $T_{e_0}^L$ sequentially for each $e_0\in\partial T_e^{R}$, while (as before) excluding edges that form cycles / loops.

We will argue that, w.h.p., for \emph{every} $e\in\vec E$, the truncated tree $T^{L + R}_e$ satisfies 
\begin{equation}\label{eq-P-R-E} 
P^R_G(e,\tilde\cE_e) > 1-2a_K\,,
\end{equation}
where $(a_k)$ is as defined in~\eqref{eq-ak-def} (so that $a_K \asymp \epsilon$), and
 \[
  \tilde\cE_e  =\left\{ e_0 \in \partial T_e^R \,:\; \tilde Q_{e_0} \leq  a_K\right\} \quad\mbox{for}\quad \tilde Q_{e_0}:=\P_G \Big( \bigcup_{t \leq L} \big\{Y_t \notin \vec E(T_e^{R;L}) \big\}\given Y_0 = e_0\Big)\,.
 \]
(That is, the NBRW started at $e_0$ does not exit  $T_e^{R;L}$, which, as opposed to the original definition of $\cE$, also precludes it from visiting the subtree $T^L_{e_1}$ for another $e_1\in\partial T_e^R$). 
The proof of~\eqref{eq-dtv-E} also yields that
\[
\E\left[ d^{(e_0)}_\tv(t_\star) \mid T_e^{R;L} \right]\one_{\{e_0\in\tilde\cE_e\}}  \leq 4a_K + o(1)\,.
\]
Hence, the upper bound will follow from~\eqref{eq-P-R-E} using the decomposition
\[ P^{t_\star+ R}_G(e,\cdot) = \sum_{e_0\in\partial T_e^R} P_G^R(e,e_0) P_G^{t_\star}(e_0,\cdot) \,.\]  
To prove~\eqref{eq-P-R-E},  note that the proof of Lemma~\ref{L:exit} further implies that
\[ \P(e_{j+1} \in \tilde\cE_e \mid T_e^R, \{T_{e_i}\}_{i=1}^j)  \geq (1-a_K)\one_{\{e_{j+1}\notin \bigcup_{i\leq j} T_{e_i}\}}
\quad\mbox{ for all $e_1,\ldots,e_{j+1} \in \partial T_e^R$}\,.\]  
The indicators of the events $\{ e_{j+1}\in\bigcup_{i\leq j} T_{e_j}\}$ ($j=1,2,\ldots,|\partial T_e^R|$) are stochastically dominated by i.i.d.\ Bernoulli random variables with parameter $n^{-1/2+o(1)}$. 

Thus, appealing to  Hoeffding--Azuma  for the random variable 
\[ S_e := \sum_{e_j\in\partial T_e^R} P_G^R(e,e_j) \one_{\{e_j \in \tilde\cE_e\}} = \P_G^R(e,\tilde\cE_e)\]
 (while noting that $P_G^R(e,e_0) \leq 2^{-R} \leq \log^{-2} n $ by the definition of $R$ and the assumption on the minimum degree assumption), we get
 \[ \P( S_e \leq 1-2a_K) \leq \exp\left( - a_K^2 \log^2 n\right)\,,\]
 which, through a union bound on $e$, establishes~\eqref{eq-P-R-E} and concludes the proof.\qed 

\subsection*{Acknowledgment}
We thank the Theory Group of Microsoft Research, where much of this research was carried out, as well as the Newton Institute. 
We are grateful to Anna Ben-Hamou and the referees for a careful reading of the paper and many significant comments and corrections on an earlier version. Indeed, one of the referee's insightful comments lead us to redefine the notion of truncated tree which plays a key role in the proofs.
N.B.~was supported in part by EPSRC grants
EP/L018896/1, EP/I03372X/1.
E.L.~was supported in part by NSF grant DMS-1513403 and BSF grant 2014361.

\bibliographystyle{abbrv}
\bibliography{RW_Gnp}

\begin{thebibliography}{10}

\bibitem{AlonSpencer08}
N.~Alon and J.~H. Spencer.
\newblock {\em The probabilistic method}.
\newblock Wiley-Interscience Series in Discrete Mathematics and Optimization.
  John Wiley \& Sons, Inc., Hoboken, NJ, third edition, 2008.

\bibitem{AN72}
K.~B. Athreya and P.~E. Ney.
\newblock {\em Branching processes}.
\newblock Dover Publications, Inc., Mineola, NY, 2004.
\newblock Reprint of the 1972 original [Springer, New York; MR0373040].

\bibitem{BHJ}
A.~D. Barbour, L.~Holst, and S.~Janson.
\newblock {\em Poisson approximation}, volume~2 of {\em Oxford Studies in
  Probability}.
\newblock The Clarendon Press, Oxford University Press, New York, 1992.
\newblock Oxford Science Publications.

\bibitem{Announce}
A.~Ben-Hamou and J.~Salez.
\newblock Cutoff for non-backtracking random walks on sparse random graphs.
\newblock {\em Ann. Probab.}
\newblock To appear.

\bibitem{BKW14}
I.~Benjamini, G.~Kozma, and N.~Wormald.
\newblock The mixing time of the giant component of a random graph.
\newblock {\em Random Structures Algorithms}, 45(3):383--407, 2014.

\bibitem{BD08}
N.~Berestycki and R.~Durrett.
\newblock Limiting behavior for the distance of a random walk.
\newblock {\em Electron. J. Probab.}, 13:no. 14, 374--395, 2008.

\bibitem{Bollobas84}
B.~Bollob{\'a}s.
\newblock The evolution of random graphs.
\newblock {\em Trans. Amer. Math. Soc.}, 286(1):257--274, 1984.

\bibitem{Bollobas-RG}
B.~Bollob{\'a}s.
\newblock {\em Random graphs}, volume~73 of {\em Cambridge Studies in Advanced
  Mathematics}.
\newblock Cambridge University Press, Cambridge, second edition, 2001.

\bibitem{DGPZ02}
A.~Dembo, N.~Gantert, Y.~Peres, and O.~Zeitouni.
\newblock Large deviations for random walks on {G}alton-{W}atson trees:
  averaging and uncertainty.
\newblock {\em Probab. Theory Related Fields}, 122(2):241--288, 2002.

\bibitem{DKLP11}
J.~Ding, J.~H. Kim, E.~Lubetzky, and Y.~Peres.
\newblock Anatomy of a young giant component in the random graph.
\newblock {\em Random Structures Algorithms}, 39(2):139--178, 2011.

\bibitem{DLP12}
J.~Ding, E.~Lubetzky, and Y.~Peres.
\newblock Mixing time of near-critical random graphs.
\newblock {\em Ann. Probab.}, 40(3):979--1008, 2012.

\bibitem{DLP14}
J.~Ding, E.~Lubetzky, and Y.~Peres.
\newblock Anatomy of the giant component: the strictly supercritical regime.
\newblock {\em European J. Combin.}, 35:155--168, 2014.

\bibitem{Durrett-RGD}
R.~Durrett.
\newblock {\em Random graph dynamics}.
\newblock Cambridge Series in Statistical and Probabilistic Mathematics.
  Cambridge University Press, Cambridge, 2010.

\bibitem{Dwass69}
M.~Dwass.
\newblock The total progeny in a branching process and a related random walk.
\newblock {\em J. Appl. Probability}, 6:682--686, 1969.

\bibitem{FR07}
N.~Fountoulakis and B.~A. Reed.
\newblock Faster mixing and small bottlenecks.
\newblock {\em Probab. Theory Related Fields}, 137(3-4):475--486, 2007.

\bibitem{FR08}
N.~Fountoulakis and B.~A. Reed.
\newblock The evolution of the mixing rate of a simple random walk on the giant
  component of a random graph.
\newblock {\em Random Structures Algorithms}, 33(1):68--86, 2008.

\bibitem{GK}
G.~Grimmett and H.~Kesten.
\newblock Random electrical networks on complete graphs {II}: Proofs.
\newblock Unpublished manuscript (1984); available at {\tt arXiv:math/0107068}.

\bibitem{JLR-RG}
S.~Janson, T.~{\L}uczak, and A.~Rucinski.
\newblock {\em Random graphs}.
\newblock Wiley-Interscience Series in Discrete Mathematics and Optimization.
  Wiley-Interscience, New York, 2000.

\bibitem{Kendall59}
D.~G. Kendall.
\newblock Unitary dilations of {M}arkov transition operators, and the
  corresponding integral representations for transition-probability matrices.
\newblock In {\em Probability and statistics: {T}he {H}arald {C}ram\'er volume
  (edited by {U}lf {G}renander)}, pages 139--161. Almqvist \& Wiksell,
  Stockholm; John Wiley \& Sons, New York, 1959.

\bibitem{Lindvall92}
T.~Lindvall.
\newblock {\em Lectures on the coupling method}.
\newblock Wiley Series in Probability and Mathematical Statistics: Probability
  and Mathematical Statistics. John Wiley \& Sons, Inc., New York, 1992.
\newblock A Wiley-Interscience Publication.

\bibitem{LK99}
L.~Lov{\'a}sz and R.~Kannan.
\newblock Faster mixing via average conductance.
\newblock In {\em Annual {ACM} {S}ymposium on {T}heory of {C}omputing
  ({A}tlanta, {GA}, 1999)}, pages 282--287. ACM, New York, 1999.

\bibitem{LS10}
E.~Lubetzky and A.~Sly.
\newblock Cutoff phenomena for random walks on random regular graphs.
\newblock {\em Duke Math. J.}, 153(3):475--510, 2010.

\bibitem{Luczak90}
T.~{\L}uczak.
\newblock Component behavior near the critical point of the random graph
  process.
\newblock {\em Random Structures Algorithms}, 1(3):287--310, 1990.

\bibitem{LPP95}
R.~Lyons, R.~Pemantle, and Y.~Peres.
\newblock Ergodic theory on {G}alton-{W}atson trees: speed of random walk and
  dimension of harmonic measure.
\newblock {\em Ergodic Theory Dynam. Systems}, 15(3):593--619, 1995.

\bibitem{LPP96}
R.~Lyons, R.~Pemantle, and Y.~Peres.
\newblock Biased random walks on {G}alton-{W}atson trees.
\newblock {\em Probab. Theory Related Fields}, 106(2):249--264, 1996.

\bibitem{NP08}
A.~Nachmias and Y.~Peres.
\newblock Critical random graphs: diameter and mixing time.
\newblock {\em Ann. Probab.}, 36(4):1267--1286, 2008.

\bibitem{Piau96}
D.~Piau.
\newblock Functional limit theorems for the simple random walk on a
  supercritical galton-watson tree.
\newblock In B.~Chauvin, S.~Cohen, and A.~Rouault, editors, {\em Trees},
  volume~40 of {\em Progress in Probability}, pages 95--106. Birkhäuser Basel,
  1996.

\bibitem{RW10}
O.~Riordan and N.~Wormald.
\newblock The diameter of sparse random graphs.
\newblock {\em Combin. Probab. Comput.}, 19(5-6):835--926, 2010.

\end{thebibliography}

\end{document}